\documentclass[centertags,12pt,a4paper]{amsart}

\usepackage{latexsym} \usepackage{amssymb,amsmath}
\usepackage{amscd} \usepackage{verbatim}

\makeatletter
\@addtoreset{equation}{section}
\makeatother

\newtheorem{theorem}{Theorem}[section]
\newtheorem{proposition}[theorem]{Proposition}
\newtheorem{conjecture}[theorem]{Conjecture}
\newtheorem{definition}[theorem]{Definition}
\newtheorem{lemma}[theorem]{Lemma}

\newtheorem*{claim*}{Claim}
\theoremstyle{definition}

\newtheorem*{question*}{Question}
\newtheorem*{theorem*}{Theorem}

\newtheorem{remark}[theorem]{Remark}

\def\cK{\mathcal{K}}

\def\cK{\mathcal{K}}

\def\cP{\mathcal{P}}

\def\cK{\mathcal{X}}

\def\fqn2{\mathbb{F}_{q^{\frac{N}{2}}}}

\def\fq{\mathbb{F}_q}

\author[]{Alexander A. Davydov}
\author[]{Massimo Giulietti}
\author[]{Stefano Marcugini}
\author[]{Fernanda Pambianco}

\thanks{2000 {\em Math. Subj. Class.}: 51E22}
\thanks{This research was performed within the activity of GNSAGA of the
Italian INDAM, with the financial support of the Italian Ministry
MIUR, project {\em Strutture Geometriche, Combinatoria e loro
Applicazioni}, PRIN 2006-2007}

\title[constructions of complete caps]{New inductive constructions of complete caps in $PG(N,q)$, $q$ even}

\begin{document}
\begin{abstract}
Some new families of small complete caps in $PG(N,q)$, $q$
even, are described. By using inductive arguments, the problem
of the construction of small complete caps in projective spaces
of arbitrary dimensions is reduced to the same problem in the
plane. The caps constructed in this paper provide an
improvement on the currently known upper bounds on the size of
the smallest complete cap in $PG(N,q),$ $N\geq 4,$ for all
$q\geq 2^{3}.$ In particular, substantial improvements are
obtained for infinite values of $q$ square, including $
q=2^{2Cm},$ $C\geq 5,$ $m\geq 3;$ for $q=2^{Cm},$ $C\geq 5,$
$m\geq 9,$ with $C,m$ odd; and for all $q\leq 2^{18}.$
\end{abstract}
\maketitle

\section{Introduction}
A cap in $PG(N,q)$, the projective $N$-dimensional space over the
finite field with $q$ elements $\fq$, is a set of points no three of
which are collinear. A cap of size $k$ is denoted as a $k$-cap. When
$N=2$, a cap is also called an arc in $PG(2,q)$.

A cap is said to be complete if it is not contained in a larger
cap. The most important problem on caps is to determine the
spectrum of possible values of $k$ for which there exists a
complete $k$-cap in $PG(N,q)$; for the known results, see
\cite{HS}, \cite{P2}, \cite{P1}, and the references therein. The smallest and the largest sizes of a complete cap are of
particular interest. This work is mainly devoted to the
construction of small complete caps that provide upper bounds
on the smallest possible size of a complete cap. New values of $k$ in the spectrum are also obtained.

Interestingly, the problem of determining the possible sizes of complete caps is related to Coding Theory. In
fact, complete $k $-caps in $PG(N,q)$ with $k>N+1$ and linear
quasi-perfect $[k,k-N-1,4]_{q}2$ -codes over $\mathbb{F}_{q}$
with covering radius $2$ are equivalent objects (with the
exceptions of the complete $5$-cap in $PG(3,2)$ giving rise to
a binary $[5,1,5]_{2}2$-code, and the complete $11$-cap in
$PG(4,3)$ corresponding to the Golay $[11,6,5]_{3}2$-code over
${\mathbb{F}}_{3}$), see e.g. \cite{GDT}, \cite{GiulPast},
\cite{HS1}.

Classical examples of
complete caps are non-singular conics for $N=2$
and elliptic quadrics for $N=3$; also, for both $N=2$ and
$N=3$, most of the known explicit constructions of complete
caps are based on subsets of points of a quadric. For $N\geq 4$
no
such natural
model for complete caps exists, a consequence of
that being the rarity of constructions of complete caps.

In this paper we describe new infinite families of complete
caps in $PG(N,q)$ for $N\geq 4$ and even $q$,
which arise as a result of some inductive procedures based on complete arcs in $PG(2,q)$.
Our main construction is described in Theorems \ref{TEO1} and
\ref{TEODD1}, see also Theorem~\ref{SINTESI} below. It should be
noted that an arbitrary complete plane arc can be taken as the
starting point for this construction; then, all known (and future)
results on the spectra of sizes of complete plane arcs (see e.g.
\cite{ABA}-\cite{DGMP}, \cite{FainaJiul}, \cite{Gi}-\cite{KV},
\cite{OS}, \cite{SZ}) provide results in higher dimension via
Theorem \ref{SINTESI}.
\begin{theorem}
\label{SINTESI} Let $q>8$, $q$ even. Assume that there exists a
complete $k$ -arc in $PG(2,q)$ with $k<q-5.$ Then there exists
a complete $n$-cap in $ PG(N,q)$ with
\begin{equation*}
n=\left\{
\begin{array}{l}
\frac{k}{q}q^{\frac{N}{2}}+3(q^{\frac{N-2}{2}}+q^{\frac{N-4}{2}}+\ldots
+q)-N+3,\quad N\geq 4\text{ even} \bigskip
\\
(2+\frac{k}{q})q^{\frac{N-1}{2}}+3(q^{\frac{N-3}{2}}+q^{\frac{N-5}{2}
}+\ldots +q)-N+4,\quad N\geq 5\text{ odd}
\end{array}
\right. \,.
\end{equation*}
\end{theorem}

Other inductive constructions presented in this paper (see Theorems \ref{TEO2}, \ref{TEO3}, \ref{TEODD2}) allow to obtain
smaller complete caps in $PG(N,q)$ from
complete $k$-arcs having some special properties,
which are connected with a new concept of "sum-points" for
a $k$-arc (see Section 2). Significantly, it
turns out that these properties are possessed by the smallest known
complete arcs in $PG(2,q)$ for any even $q\leq 2^{17}$ (see
Table 1), by the complete $k$-arcs of \cite{Gi} with $k\leq (q+4)/2$,
and by the Abatangelo complete $ (q+8)/3$-arcs of \cite{ABA}
(see Lemmas \ref{Lem_q/2+2} and \ref {Lem_Abatangelo}).


As a consequence of our results, substantial improvements on
the known bounds on the size $t_{2}(N,q)$ of the smallest
complete cap in $PG(N,q)$ (or, equivalently, the minimal length
$k$ for which there exists an $ [k,k-N-1,4]_{q}2$-code) are
obtained for $q\geq 8$.
In \cite{PS} it was proved that
\begin{equation}
t_{2}(N,q)\leq \left\{
\begin{array}{l}
q^{\frac{N}{2}}+s_{N,q},\,\,\,N\,\,\text{even}\medskip \\
(2+1)q^{\frac{N-1}{2}}+s_{N,q},\,\,\,N\,\,\text{odd}
\end{array}
\right. ,  \label{PSB}
\end{equation}
where
\begin{equation*}
s_{N,q}=3(q^{\lfloor \frac{N-2}{2}\rfloor }+q^{\lfloor \frac{N-2}{2}\rfloor
-1}+\ldots +q)+2\,.  \label{esse}
\end{equation*}
For about a decade, (\ref{PSB}) was the best known upper bound
on $t_{2}(N,q) $, $q>2$, $N>3$, with the exceptions of few
small values of both $N$ and $q$ , see \cite{HS}, \cite{P2},
\cite{P1}, \cite{OS}. Inequality~(\ref{PSB}) was improved in
\cite{Gi}:
\begin{equation}
t_{2}(N,q)\leq \left\{
\begin{array}{l}
\frac{1}{2}q^{\frac{N}{2}}+s_{N,q},\,\,\,N\,\,\text{even} \medskip \\
(2+\frac{1}{2})q^{\frac{N-1}{2}}+s_{N,q},\,\,\,N\,\,\text{odd}
\end{array}
\right. ,\,q\geq 32.  \label{giu2}
\end{equation}
Better upper bounds were obtained for specific values of
$q\leq2^{15}$, see also \cite{GiulPast}. Also,
in~\cite{GiulPast} it was proved that
\begin{equation}
t_{2}(N,q)\leq \frac{1}{3}q^{\frac{N}{2}}+\frac{5}{3}q^{\frac{N-2}{2}
}+s_{N,q}-2,\quad q\geq 2^{8}\text{ square},\,\,\,N\geq 4\text{ even}.
\label{giupast}
\end{equation}

All the above upper bounds on $t_{2}(N,q)$ are improved in this
paper. Theorems \ref{NEW1} and \ref{NEW4} yield the following inequality.
\begin{itemize}
\item
For $q>8$, $q$ even,
\begin{equation}
\quad t_{2}(N,q)\leq \left\{
\begin{array}{l}
\frac{t_{2}(2,q)}{q}q^{\frac{N}{2}}+s_{N,q}-N+1,\quad N\text{ even } \bigskip
\\
(2+\frac{t_{2}(2,q)}{q})q^{\frac{N-1}{2}}+s_{N,q}-N+2,\quad N\text{ odd }
\end{array}
\right. .  \label{apr07}
\end{equation}
\end{itemize}



Then, each upper bounds on $t_{2}(2,q)$ for $q$
even gives rise to an upper bound on $t_{2}(N,q)$.
For instance, from  \cite[Remark
2]{SZ} the following inequalities are obtained.
\begin{itemize}
\item For $q=2^{2Cm}\geq 2^{30}$, $m\geq 2$, $C\geq 5$,
\begin{equation}
t_{2}(N,q)\leq \left\{
\begin{array}{l}
\frac{1}{2^{m-1}}q^{\frac{N}{2}}+s_{N,q}-N+1,\,\,\,N\,\,\text{even}
\medskip \\
(2+\frac{1}{2^{m-1}})q^{\frac{N-1}{2}}+s_{N,q}-N+2,\,\,\,N\,\,\text{odd}
\end{array}
\right. .  \label{form_Intr_Szon1}
\end{equation}
\item For $q=2^{Cm}\geq 2^{30}$, $m\geq 2$, $C\geq 5$,
\begin{equation}
t_{2}(N,q)\leq \left\{
\begin{array}{l}
\frac{1}{2^{(m/3)-1}}q^{\frac{N}{2}}+s_{N,q}-N+1,\,\,\,N\,\,\text{even} \medskip \\
(2+\frac{1}{2^{(m/3)-1}})q^{\frac{N-1}{2}}+s_{N,q}-N+2,\,\,\,N\,\,\text{odd}
\end{array}
\right. .  \label{form_Intr_Szon2}
\end{equation}
\end{itemize}
For $q=2^{14},2^{18}$ we can take into account the bounds
$t_{2}(2,q)\leq 6(\sqrt{q} -1),$ see \cite{DGMP}.

Other results arise by
Theorems \ref{NEW2}, \ref{NEW3}, and \ref{NEW5}, together with Table 1 and Lemmas
\ref{Lem_q/2+2} and \ref{Lem_Abatangelo}.

\begin{itemize}
\item For $q\geq 32,$
\begin{equation}
t_{2}(N,q)\leq \left\{
\begin{array}{l}
\frac{1}{2}q^{\frac{N}{2}}+\frac{5}{6}s_{N,q}+\frac{1}{3},\quad N\geq 4\text{
even} \bigskip
\\
(2+\frac{1}{2})q^{\frac{N-1}{2}}+\frac{5}{6}s_{N,q}+\frac{1}{3},\quad N\geq 5
\text{ odd}
\end{array}
\right. .  \label{form_Intr_a}
\end{equation}
\item For $q\geq 2^{6}$ square,
\begin{equation}
t_{2}(N,q)\leq \left\{
\begin{array}{l}
\frac{1}{3}q^{\frac{N}{2}}+s_{N,q}+\frac{2}{3},\quad N\geq 4\text{ even}\bigskip
\\
(2+\frac{1}{3})q^{\frac{N-1}{2}}+s_{N,q}+\frac{2}{3},\quad N\geq 5\text{ odd}
\end{array}
\right. .  \label{form_Intr_b}
\end{equation}

\item For $q\leq 2^{15}$,
\begin{equation*}
\quad t_{2}(N,q)\leq \left\{
\begin{array}{l}
\frac{t_{q}}{q}q^{\frac{N}{2}}+\frac{t_{q}}{3q}(s_{N,q}-2),\quad N\geq 4
\text{ even}\bigskip  \\
(2+\frac{t_{q}}{q})q^{\frac{N-1}{2}}+\frac{t_{q}}{3q}(s_{N,q}-2),\quad N\geq
5\text{ odd}
\end{array}
\right. \,.  \label{estim_sum-point}
\end{equation*}
where  $t_{q}$ is as in the following table.
$$
\begin{array}{|c|c|c|c|c|c|c|c|c|c|c|c|c|c|}
\hline
\log _{2}q & 3 & 4 & 5 & 6 & 7 & 8 & 9 & 10 & 11 & 12 & 13 & 14 & 15 \\
\hline
t_{q} & 6 & 9 & 14 & 22 & 34 & 55 & 86 & 124 & 201 & 307 & 461 & 665 & 1026
\\ \hline
\end{array}
$$
\end{itemize}



It is easy to see that these new upper bounds improve the known
bounds on $t_2(N,q)$ for any $q\ge 8$ and any dimension $N\ge 4$. In
order to assess the above improvements, we introduce and discuss two
parameters, $\Delta_{N,q}$ and $R_{N,q}$. Define $\Delta _{N,q}$ as
the difference between the best known upper bounds on $t_{2}(N,q)$
and the new bounds obtained in this work. By $R_{N,q}$ we denote the
ratio between the coefficients of the main term $q^{\left\lfloor
\frac{N}{2}\right\rfloor }$ in our upper bounds and  in the best
known ones.

If $q$ is not a square, then $ \Delta _{N,q}\gtrsim
\frac{1}{6}s_{N,q}$; if in addition  $q=2^{Cm}$, $m\geq 9$, $C\geq
5$, then  $\Delta _{N,q}\gtrsim
(\frac{1}{2}-\frac{1}{2^{(m/3)-1}})q^{\left\lfloor \frac{N}{2}
\right\rfloor }$. If $q$ is a square, then $ \Delta _{N,q}\gtrsim
\frac{5}{3}q^{\frac{N-2}{2}}$ for $N$ even, and $\Delta
_{N,q}\gtrsim \frac{1}{6}q^{\frac{N-1}{2}}$ for $N$ odd.
Furthermore, if $q=2^{2Cm}$, $m\geq 3$, $C\geq 5$, then $\Delta
_{N,q}\gtrsim (\frac{1}{3}-\frac{1}{2^{m-1}})q^{\frac{N}{2}}+
\frac{5}{3}q^{\frac{N-2}{2}}$ for $N$ even, and $\Delta
_{N,q}\gtrsim ( \frac{1}{2}-\frac{1}{2^{m-1}})q^{\frac{N-1}{2}}$ for
$N$ odd.

If $q$ is a square then $%
R_{N,q}\approx \frac{14}{15}$ for $N$ odd. If $q=2^{2Cm}$,
$m\geq 3$, $C\geq 5$, then $R_{N,q}\approx \frac{3}{2^{m-1}}$
for $N$ even, and $ R_{N,q}\approx \frac{2^{m}+1}{5\cdot
2^{m-2}}$ for $N$ odd. If $q=2^{Cm}$, $ m\geq 9$, $C\geq 5,$
$m,C$ odd$,$ then $R_{N,q}\approx \frac{1}{2^{(m/3)-2}}$ for
$N$ even, and $R_{N,q}\approx \frac{2^{(m/3)}+1}{5\cdot
2^{(m/3)-2}}$ for $N$ odd.


For $2^{3}\leq q\leq 2^{18},$ we list only some of the values
of the new upper bounds on $t_{2}(N,q)$ obtained in this work,
and those of the corresponding $\Delta_{N,q}$. In each entry $\Delta_{N,q}$ of the
table, we cite the paper where the best previously known bound was
proved.
\begin{equation*}
\begin{array}{|c||@{}r@{\,}|r||@{}r@{\,}|r||@{}r@{\,}|r||@{}r|r|}
\hline
N &
\begin{array}{@{\,}c@{}}
t_{2}(N,2^{3})
\end{array}
& \Delta _{N,2^{3}} &
\begin{array}{@{\,}c@{}}
t_{2}(N,2^{4})
\end{array}
& \Delta _{N,2^{4}} &
\begin{array}{@{\,}c@{}}
t_{2}(N,2^{5})
\end{array}
& \Delta _{N,2^{5}} &
\begin{array}{@{\,}c@{}}
t_{2}(N,2^{6})
\end{array}
& \Delta _{N,2^{6}} \\ \hline
4 & \leq 54 & 18\text{ \cite{OS}} & \leq 153 & 25\text{ \cite{GiulPast}} & \leq 462 & 148
\text{ \cite{Gi}} & \leq 1430 & 812\text{ \cite{Gi}} \\ \hline
5 & \leq 182 & 36\text{ \cite{PS}} & \leq 665 & 153\text{ \cite{PS}} & \leq 2510 & 148\text{
\cite{Gi}} & \leq 9622 & 812\text{ \cite{Gi}} \\ \hline
6 & \leq 438 & 292\text{ \cite{PS}} & \leq 2457 & 409\text{ \cite{GiulPast}} & \leq 14798 &
4756\text{ \cite{Gi}} & \leq 91542 & 52012\text{ \cite{Gi}} \\ \hline
\end{array}
\end{equation*}

Some comparisons are also given  after Theorems \ref{TEO3} and
\ref{TEODD2}.


The paper is organized as follows. Section \ref{s2} contains some
preliminary ideas and results on $k$-arcs in $PG(2,q)$. The concept
of sum-points for a $k$-arc is introduced, and $k$-arcs with only
one sum-point are investigated. These arcs will be the base for some
of the inductive constructions of complete caps that are described
in Section \ref{main} for $N$ even, and Section \ref{secODD} for $N$
odd. Most of these caps are such that the intersection with an
$M$-dimensional subspace of $PG(N,q)$ is a cap of the same type
$K^{(M)}_{m_1,m_2}$, see (\ref{m1m2}). Being quite technical, the
investigation of caps of type $K^{(M)}_{m_1,m_2}$ is postponed in
the Appendix.

\section{Sum-points for plane arcs}\label{s2}
Throughout the paper, $q$ is a power of $2$. Let  $\fq$ denote the
finite field with $q$ elements,  and let $\fq^{\ast }=\fq\setminus
\{0\}$. Let $X_{0},X_{1},X_{2}$ denote homogeneous coordinates for
points of $PG(2,q)$.

In this section we prove some preliminary results on plane arcs. Let
$l_\infty$ be the line of $PG(2,q)$ of equation $X_0=0$.  The points
of an  arc $K$  not lying on $l_\infty$ are the \emph{affine points
} of $K$, and the subset of affine points of $K$ is the \emph{affine
part }of $K$. An arc is said to be {\em affine} if it coincides with
its affine part. An {\em affinely complete arc} is an affine arc
whose secants cover all the points in $PG(2,q)\setminus l_\infty$.
We recall that an $R$-secant of $K$ is a line $l$ such that $\mid
l\cap K\mid =R$. As usual, we say that a point is written in his
normalized form if the first nonzero coordinate is equal to $1$.

Let $K$ be a complete arc in $PG(2,q)$, and let $Q$ be a point in
$PG(2,q)\setminus K$ written in its normalized form. For every
secant $l$ of $K$ through $Q$, let $c_1^{(l)},c_2^{(l)}$ be the
elements in  $\fq^\ast$ such that
$$
Q=c_1^{(l)}P_1+c_2^{(l)}P_2
$$
where $P_1$ and $P_2$ are the points on $l\cap K$ written in their
normalized form.
\begin{definition} The point $Q$ is said to be a {\em sum-point} for
$K$ if $c_1^{(l)}=c_2^{(l)}$ for every secant $l$ of $K$ through
$Q$.
\end{definition}

%

\begin{remark}
In general, collineations do not preserve the number of sum-points
for an arc. In this sense the concept of \textquotedblleft
sum-points\textquotedblright\ is \textquotedblleft not
geometrical\textquotedblright .
\end{remark}

We denote by $\beta (K)$ the number of sum-points for a complete arc
$K$. When $\beta (K)=1$, we denote by $p(K)$  the number of secants
of $K$ passing through the only sum-point.


\begin{lemma}
\label{Lem_beta(K)>=1}
 Let $K$  be a complete arc in
$PG(2,q)$. Then $\beta(K)\ge 1$.
\end{lemma}

\begin{proof} Let $K$ be a complete arc in $PG(2,q)$. Note that, as $q$ is
even,
$$
(0,m,m')=c_1(1,a,b)+c_2(1,a',b'),\qquad (1,a,b)\neq (1,a',b'),
$$
yields  $c_1=c_2$. Therefore, if $l_\infty$ is either a $0$-secant
or a $1$-secant of $K$, then every point in $l_\infty\setminus K$ is
a sum-point. Assume then that $l_\infty$ is a $2$-secant, and let
$K\cap l_\infty=\{(0,X_1,X_2),(0,1,f)\}$, where $X_1$ is either $0$ or $1$. Then the point of
coordinates $(0,X_1+1,X_2+f)$ is clearly a sum-point, which proves
the assertion.
\end{proof}

\begin{lemma}
\label{Lem_affin+2} Let $K$ be a complete arc in $PG(2,q)$ such that
$l_\infty$ is a secant of $K$. If the affine part of $K$ is affinely
complete then $\beta (K)=1.$
\end{lemma}

\begin{proof}
We first prove that  any point $P=(1,x,y)$ is not a sum-point for
$K$. Two distinct affine points of $K$, say $(1,a,b)$ and
$(1,a',b')$, are collinear with $P$, that is,
$$(1,x,y)=c_1
(1,a,b)+c_2(1,a',b'),$$ for some  $c_1,c_2\in \fq^\ast$. As $q$ is
even, it is not possible that $c_1=c_2$. Now let $K\cap
l_\infty=\{(0,X_1,X_2),(0,1,f)\}$, with $X_1\in \{0,1\}$. Then it is straightforward that
the point of coordinates $(0,X_1+1,X_2+f)$ is the only sum-point for
$K$ on $l_\infty$.
\end{proof}

\begin{remark}
The converse of Lemma \ref{Lem_affin+2} does not hold, as it can
be shown that there exist arcs $K$ in $PG(2,q)$, $q=8,16$, such
that $\beta (K)=1$ but the secants of the affine part of $K$ do
not cover all the affine points of $PG(2,q)$, see Table 1.
\end{remark}

For a $k$-arc $K$ in $PG(2,q)$, let
\begin{eqnarray*}
Cov_\infty(K)=\{m\mid (0,1,m) \text{ is covered by the secants of }
K\},\qquad\qquad \\
S_\infty(K)=\{X_2+Y_2\mid (X_0,X_1,X_2),(Y_0,Y_1,Y_2)\in K,\,\,
X_0=Y_0,X_1=Y_1,X_2\neq Y_2\}.
\end{eqnarray*}
Also, for any element $m\in \fq$, let
$$
S_m(K)=\left\{X_1+Y_1\mid (X_0,X_1,X_2),(Y_0,Y_1,Y_2)\in K,\,\,
X_0=Y_0,X_1\ne Y_1, \frac{X_2+Y_2}{X_1+Y_1}=m\right\}.
$$
Note that the size of $S_m(K)$ is at most $k/2$, as there are
at most $k/2$ pairs of points of  $K$ collinear with $(0,1,m)$.
The size of $S_\infty(K)$ is at most $k/2$ as well, as
$S_\infty(K)$ corresponds to pair of points of $K$ collinear
with $(0,0,1)$. In particular, when $K$ is complete and $\beta
(K)=1$ and the only sum-point for $K$ is $(0,0,1)$ we
have
\begin{equation*}
1\leq |S_{\infty }(K)|\leq p(K).
\end{equation*}
Note also that $0\notin S_{\infty }(K).$ Similarly, if $\beta (K)=1$
and the only sum-point for $K$ is $(0,1,m)$ we have
\begin{equation*}
1\leq |S_{m}(K)|\leq p(K).
\end{equation*}

As a matter of terminology, we say that a projectivity $\psi$ of
$PG(2,q)$ is {\em integral for }$K$ if it can be represented by a
matrix $A\in GL(3,q)$ with the following property: for each point
$P$ of $K$ written in its normalized form, $A\cdot P$ is the
normalized form for the point $\psi(P)$.
\begin{lemma}
\label{Lem_!!!!} Let $K$ be a complete arc in $PG(2,q)$. Let $\psi$
be any  projectivity  of $PG(2,q)$ which is integral for $K$. Then a
point  $Q\in PG(2,q)\setminus K$ is a sum-point for $K$ if and only
if $\psi(Q)$ is a sum-point for $\psi(K)$. In particular,
$\beta(\psi(K))=\beta(K)$.
%
%
\end{lemma}
\begin{proof}
Assume that $Q$ is not a sum-point for $K$. Then there exists
$c_1\neq c_2$ such that
$$
Q=c_1P_1+c_2P_2,
$$
where $P_1$ and $P_2$ are points of $K$ written in their normalized
form. Let $A$ be a matrix representing $\psi$ and such that for each
point $P$ of $K$ written in its normalized form, $A\cdot P$ is the
normalized form for  $\psi(P)$. Then
$$
\psi(Q)=A\cdot Q=c_1(A\cdot P_1)+c_2 (A\cdot P_2),
$$
whence $\psi(Q)$ is not a sum-point for $K$. The converse can be
proved in a similar way.
\end{proof}
\begin{remark}
Let $K$ be a complete arc such that  $\beta(K)=1$. From the proof of
Lemma \ref{Lem_!!!!} it follows that for any projectivity $\psi$
which is integral for $K$, the value of $p(\psi(K))$  coincides with
$p(K)$.
\end{remark}
\begin{lemma}
\label{Lem_001} For every complete arc $K$ in $PG(2,q)$ with $\beta
(K)=1$ there is a projectivity  $\psi$  such that $\beta
(\psi(K))=1$,   $p(\psi(K))=p(K)$, and the only sum-point for
$\psi(K)$ is $(0,0,1).$
\end{lemma}

\begin{proof}
As $\beta (K)=1,$ the line $l_\infty$ is a secant of $K$. Let
$K\cap l_\infty=\{(0,X_{1},f),(0,1,g)\}$. If $X_{1}=1$, the
lemma is proved by taking $\psi$ as the identical projectivity.
Assume then that $X_{1}=0.$ Then $f=1.$ Let $\psi
(x,y,z)=(x,y(g+1)+z,$ $z+gy)$. Clearly $\psi$ is integral for
$K$. Also, $\psi (K)\cap l_\infty=\{(0,1,1),(0,1,0)\}$, whence
$(0,0,1)$ is a sum-point for $\psi(K)$. By Lemma
\ref{Lem_!!!!}, the assertion is proved.
\end{proof}

\begin{lemma}
\label{Lem_011}For every complete arc $K$ in $PG(2,q)$ with $\beta
(K)=1$ there exists a projectivity $\psi$ such that
$\beta(\psi(K))=1$, $\psi(K)\cap l_\infty=\{(0,0,1),(0,1,0)\}$,
$p(\psi(K))=p(K)$, and the only sum-point for $\psi(K)$ is
$(0,1,1).$
\end{lemma}
\begin{proof}
As $\beta (K)=1,$ the line $l_\infty$ is a secant of $K$. Let
$K\cap l_\infty=\{(0,X_{1},f),(0,1,g)\}$. If $X_{1}=0$ then
$f=1$, whence the assertion holds for $\psi
(x,y,z)=(x,y,z+gy)$. Assume then that $X_{1}=1.$ Then let
$\psi(x,y,z)=(x,
\frac{f}{f+g}y+\frac{1}{f+g}z,\frac{g}{f+g}y+\frac{1}{f+g}z).$
By Lemma \ref{Lem_!!!!}, the claim follows.
\end{proof}

\begin{lemma}
\label{Lem_Sinfnot1} For every complete arc $K$ in $PG(2,q)$ with
$\beta (K)=1$  there is a projectivity $\psi$  such that $\beta
(\psi(K))=1$,   $p(\psi(K))=p(K)$,  the only sum-point for $\psi(K)$
is $(0,0,1)$, and $1\notin S_{\infty }(\psi(K))$.
\end{lemma}
\begin{proof}
By Lemma \ref{Lem_001} we can assume that the only sum-point
for $K$ is $(0,0,1)$. Then $K\cap l_\infty=\{(0,1,f),(0,1,g)\}$
for some $f,g\in \fq$. For a pair $ P_1=(X_{0},X_{1},X_{2}),$
$P_2=(Y_{0},Y_{1},Y_{2})$ of points of $K$ collinear with
$(0,0,1),$ let $\Delta _{P_1,P_2}=X_{2}+Y_{2}$. As there are at
most $k/2$ distinct values of $\Delta _{P_1,P_2},$ there exists
an element $w\in \fq^{\ast }$ such that
\begin{equation*}
w\notin \{\Delta _{P_1,P_2}|P_1,P_2\text{ collinear with }(0,0,1)\}.
\end{equation*}
Let $\psi (X_{0},X_{1},X_{2})=(X_{0},X_{1},\frac{X_{2}}{w})$.
Then $\psi (K)\cap
l_\infty=\{(0,1,\frac{f}{w}),(0,1,\frac{g}{w})\}$. Note that
$\Delta _{\psi (P_1),\psi (P_2)}=(1/w)\Delta_{P_1,P_2}\neq 1,$
for every pair $\psi(P_1),\psi(P_2)$ of points collinear with
$(0,0,1)$, whence $1\notin S_{\infty }(\psi(K)).$ Also, the
point $(0,0,1)$ is a sum-point for $\psi(K)$. Finally, Lemma
\ref{Lem_!!!!} ensures that $\beta (\psi(K))=\beta (K)=1$.
\end{proof}

\begin{lemma}
\label{Lem_forLaScala}In $PG(2,q)$ for every complete $k$-arc
$K$ with $\beta (K)=1$ and\begin{equation} (k-2)p(K)<q-1
\label{form_(k-2)p(K)}
\end{equation}there exists a collineation $\psi$ such that $\psi(K)\cap
l_\infty=\{(0,0,1),(0,1,0)\}$,  $\beta (\psi(K))=1$,
$p(\psi(K))=p(K)$,  the only sum-point for $\psi(K)$ is $(0,1,1),$
and with the property that
\begin{equation}
\psi(K)\cap \{(1,a ,Aa ^{2})|A\in S_{1}(\psi(K)),a \in \fq\}
=\emptyset . \label{form_forLaSc1}
\end{equation}
\end{lemma}

\begin{proof}
By Lemma \ref{Lem_011}, we can assume that $K\cap
l_\infty=\{(0,0,1),(0,1,0)\}$ and $\beta (K)=1,$  the only sum-point
for $K$ being $(0,1,1).$ Let $K_{w}=\phi _{w}(K)$ where $\phi
_{w}(x,y,z)=(x,wx+y,z),$ $w\in \fq^{\ast }.$ Note that $\beta
(K_{w})=\beta (K)=1$ by Lemma \ref{Lem_!!!!}. As $\phi_{w}(K)\cap
l_\infty=\{(0,0,1),(0,1,0)\}$, it follows that
 the only sum-point for $K_{w}$ is $(0,1,1).$ Also, it is straightforward that
$S_{1}(K_{w})=S_{1}(K)$.

Now let $R=\{(1,a ,Aa ^{2})|A\in S_{1}(K),a \in \fq\}.$ Note that
for any affine point $P=(1,b,c)$ in $PG(2,q)$, the point $\phi_w(P)$
belongs to $R$ if and only if $w^2=\frac{1}{A}(c+Ab^2)$. When $P$
ranges over the affine points of $K$, and $A$ over the set
$S_{1}(K)$, the number of values $\frac{1}{A}(c+Ab^2)$ is at most
$(k-2)|S_{1}(K)|$. This proves that
 $K_{w}\cap R \neq \emptyset$ for at most
 $(k-2)|S_{1}(K)|$ values of $w$.   As
$|S_{1}(K)|\leq p(K)$, from (\ref{form_(k-2)p(K)}) it follows that
there exists $w_0\in \fq^\ast$ such that  $K_{w_0}\cap R=\emptyset$.
Then the assertion follows for $\psi=\phi_{w_0}$.
\end{proof}

\begin{lemma}
\label{Lem_q/2+2} For any even $q\geq 32,$ in $PG(2,q)$ there
exists a complete $k$-arc $K$ with $k\leq (q+4)/2,$ $\beta (K)=1$,
$p(K)=1.$
\end{lemma}

\begin{proof}
By \cite[Proposition 3.2, Lemmas 4.1-4.3]{Gi}, for any even
$q\geq 32,$ in $PG(2,q)$ there exists a complete $(\frac{q}{\alpha
}+2)$-arc $K$, $\alpha \geq 2$ integer, which is
obtained from an affinely complete $\frac{q}{\alpha }$-arc
$K^{A}$ by adding two points lying on the line $l_{\infty }.$
Moreover, the number of points on $\ell _{\infty }$ covered by
$K^{A}$ is $\frac{q}{\alpha }-1$. This means that there are at
least $3$ points on $\ell _{\infty }$ uncovered by $K^{A}$.
Such points can be assumed to be $(0,1,f)$, $(0,1,g),$ and
$(0,0,1)$. Then it is straightforward that $K$ is a complete
arc with the only sum-point $(0,0,1)$ (see also the proof of Lemma
\ref{Lem_affin+2}).  As $K^{A}$ does not cover
the sum-point $(0,0,1)$, $p(K)=1$ holds.
\end{proof}

\begin{lemma}
\label{Lem_Abatangelo} Let $q=2^{h},$ $h\geq 6$ even. Then there exists
a complete $(q+8)/3$-arc $K$ in $PG(2,q)$ such that $\beta
(K)=1$, $p(K)=1.$
\end{lemma}

\begin{proof}
The following construction comes from \cite{ABA}.
Let $g$ be a primitive element of the field $\mathbb{F}_{q}.$
The points $Y_{1},$ $Y_{2},$ and $X_{\infty },$ the pointsets
$C_{3},$ $D_{3}$, $F_{3},$ and $K,$ are defined as follows in \cite{ABA}:
$Y_{1}=(0,1,g^{-1}),$ $Y_{2}=(0,1,g^{-2}),$ $X_{\infty
}=(1,0,0),$ $C_{3}=\{(1,g^{3r},g^{-3r})\,| \,r=0,1,\ldots
,\frac{q-4}{3}\},$ $D_{3}=\{(1,g^{1-3r},g^{-(1-3r)})\,|
\,r=0,1,\ldots ,\frac{q-4}{3}\}, F_{3}=\{(1,g^{2-3r},g^{-(2-3r)})\,| \,r=0,1,\ldots
,\frac{q-4}{3}\},
K=C_{3}\cup \{Y_{1},Y_{2},X_{\infty }\}$.
In \cite{ABA} the following
assertions are proved:
\begin{itemize}
\item[\textrm{(i)}] every point in
    $T=PG(2,q)\smallsetminus (C_{3}\cup D_{3}\cup
    F_{3}\cup l_{\infty }\cup \{(0,0,1),(0,1,0),(1,0,0)\})
    $ lies on some bisecant of $C_{3}$;
\item[\textrm{(ii)}] for every point $D\in D_{3}$ (resp.,
    $F\in F_{3}$) there is a point
$P_{r}=(1,g^{3r},g^{-3r})\in C_{3}\,$such that the points
$D,P_{r},$ and $Y_{1}$ (resp., $F,P_{r},$ and $Y_{2}$) are
collinear;
\item[\textrm{(iii)}] $K$ is a complete $(\frac{q+8}{3})-$arc.
\end{itemize}
By \cite[Chapter 4, Corollary~28]{MWS}, $g$ can be assumed to have  trace equal to $1$.
Note that $l_{\infty }$ is a 2-secant of $K$ and $(0,0,1)$ is
the only sum-point lying on $l_{\infty }$, see the proof of
Lemma~\ref{Lem_beta(K)>=1}. Sum-points not belonging to
$l_{\infty }$ can only be of the form $S_{r,i}=P_{r}+Y_{i},$
$i\in \{1,2\}.$ In fact, sum-points of the form $X_{\infty
}+Y_{i}$ do not exist as, by (i) and (ii), every point of
$PG(2,q)\smallsetminus K$ lies on some bisecant of
$K\smallsetminus \{X_{\infty }\}.$ If $S_{r,i}\in T$ then, by
(i), $S_{r,i}$ lies on a bisecant of $C_{3}$ and then it is not
a sum-point, see the proof of Lemma \ref{Lem_affin+2}. If
$S_{r,i}\notin T$ then $S_{r,i}\in D_{3}\cup F_{3}.$ This means
that $(1,g^{3r},g^{-3r})+(0,1,g^{-i})=(1,a,a^{-1}),$ $a\in
\mathbb{F}_{q},$ whence $(g^{3r})^{2}+g^{3r}+g^{i}=0$ for some
$i\in \{1,2\}$. But this is impossible as the  traces of $g$
and $g^{2}$ are both equal to one.

Therefore,  the only
sum-point is $(0,0,1).$ As it can be obtained only as
the sum $Y_{1}+Y_{2}$, $p(K)=1$ holds.
\end{proof}

Let $\overline{t}_{2}(2,q)$ and $\overline{t}_{2}^{A}(2,q)$ be
the smallest \emph{known }size of a complete arc and of an
affinely complete arc in $ PG(2,q)$, respectively. As every
complete arc is projectively equivalent to an affine arc, we
have $\overline{t}_{2}^{A}(2,q)\leq \overline{t}_{2}(2,q)\leq
\overline{t}_{2}^{A}(2,q)+2.$ Let $\overline{t}_{2}^{\ast
}(2,q)$ be the smallest \emph{known }size of a complete arc $K$
in $PG(2,q)$ with $\beta (K)=1$, and let $\overline{p}(2,q)$ be
the smallest \emph{known } value of $p(K)$ for arcs of size
$\overline{t}_{2}^{\ast }(2,q)$ with $\beta (K)=1$. For $q\leq
2^{18}$, the values of these parameters, either known in the
literature or obtained in this work, are listed in Table 1. Any
value in the table which is not only the smallest \emph{known},
but also the smallest \emph{possible}, is followed by a dot. It
should be noted that for some $q$ we have
$\overline{t}_{2}^{A}(2,q)<\overline{t}_{2}(2,q)$. For $
q=2^{3},2^{4},2^{5},2^{6},2^{9},$ we use the previously known
small complete arcs. For $q=2^{7},2^{15},2^{17},$ by
\cite[Section 5, Lemma 4.3]{Gi}, complete
$\overline{t}_{2}(2,q)$-arcs $K$ are obtained from affinely
complete $\overline{t}_{2}^{A}(2,q)$-arcs $K^{A}$ by adding two
points lying on $l_{\infty }$, as in the proof of Lemma
\ref{Lem_q/2+2}. By the same arguments, the values of
$\overline{t}_{2}^{\ast }(2,2^{18})$ and $\overline{p
}(2,2^{18})$ are obtained. The value of
$\overline{t}_{2}(2,2^{18})$ comes from \cite{DGMP}, where
$6{(\sqrt{q}-1)}${-arcs} in $PG(2,q),$ $q=4^{2h+1},$ are
constructed and for $h\leq 4$ it is proved that they are
complete. Unfortunately, nothing is known on $\beta (K)$ for
these arcs. For $q=2^{16}$ the entries follow from \cite{ABA}
and Lemma \ref{Lem_Abatangelo}.

For $q=2^{8},2^{10},2^{11},2^{12},2^{13},2^{14}$, the complete
$\overline{t}
_{2}(2,q)$-arcs in Table 1 are new.
For $q=2^{11},2^{13},$ they have been obtained by using the
randomised greedy algorithms \cite{P1}, \cite{P2} with a random
starting set. For $q=2^{8},2^{10},2^{12},2^{14},$
our arcs are based on
the $(4\sqrt{q}-4)$-arcs $\mathcal{K}_{w}$ in
$PG(2,q)$ introduced in \cite[p. 115]{FainaJiul}:
\begin{equation}
\mathcal{K}_{w}=\{(1,1/\alpha ,\alpha),(1,1/w\alpha ,w\alpha),(1,w^{
\sqrt{q}-1}/\alpha,\alpha),(1,1/w\alpha ,w^{\sqrt{q}}\alpha)|\alpha \in \mathbb{
F}_{\sqrt{q}}^{\ast }\}  \label{form2_Kw}
\end{equation}
where $w$ is an element of $\mathbb{F}_{q}\smallsetminus \mathbb{F}_{\sqrt{q}%
}$ satisfying $w^{2}+w+d=0,$ with $d\in \mathbb{F}_{\sqrt{q}}$.
Let $\gamma $ be a primitive element of $\mathbb{F}_{q}.$ For
$q=2^{12},$ we put $d=\gamma ^{(q-1)/3}$ and use
$\mathcal{K}_{w}$ as the starting set for the greedy
algorithms. For $q=2^{8},$ the starting set is a subset of
$\mathcal{K}_{w}.$ For $q=2^{10},2^{14},$ we modify
$\mathcal{K}_{w}$.

\begin{lemma}
\label{Lem_Kw'} Let $q=4^{2h+1},$ $h\geq 1.$ Let $w=\gamma
^{(q-1)/3}.$
Then $w^{2}+w+1=0.$
Let $\mathcal{K}_{w}^{\prime
}$ be the point set obtained from the arc $\mathcal{K}_{w}$ of
{\rm (\ref{form2_Kw})} by changing every point $(1,1/w\alpha
,w^{\sqrt{q}}\alpha )$ with the point $(1,1/w^{2}\alpha
,w^{2\sqrt{q}}\alpha)$. Then $\mathcal{K}_{w}^{\prime }$ is a
$(4\sqrt{q}-4)$-arc in $PG(2,q),$ $q=2^{6},2^{10},2^{14}$. For
$q=2^{6},2^{10},$ the arc is complete.
\end{lemma}

\begin{proof}
The assertion about $w$ is trivial. The properties of
$\mathcal{K}_{w}^{\prime }$ have been checked by computer.
\end{proof}

If $q=2^{14}$ the arc $\mathcal{K}_{w}^{\prime }$ turns out not to be complete, and in order to obtain a complete arc in $PG(2,2^{14})$ we use
$\mathcal{K} _{w}^{\prime }$ as the starting point for the greedy
algorithms.

Finally, it should be noted that for $q=2^{3},\ldots
,2^{6},2^{8},\ldots ,2^{14}$, the values of $\overline{t}
_{2}^{\ast }(2,q)$ and $\overline{p}(2,q)$ have been obtained
by acting on complete $\overline{t}_{2}(2,q)$-arcs with both
randomly chosen collineations and  projectivities of
type $\phi (X_{0},X_{1},X_{2})=(X_{0},X_{1}+wX_{2},X_{2}).$

\newpage

\begin{center}
Table 1. Parameters for $q\leq 2^{18}$
\end{center}
\begin{equation*}
\renewcommand{\arraystretch}{1.0}
\begin{array}{c|c|c|c|c|c|l}
\hline
q & \overline{t}_{2}^{A}(2,q) & \overline{t}_{2}(2,q) & \overline{t}
_{2}^{\ast }(2,q) & \overline{p}(2,q) & (\overline{t}_{2}^{\ast }(2,q)-2)
\overline{p}(2,q) & \text{References} \\ \hline
2^{3} & 6\centerdot  & 6\centerdot  & 6\centerdot  & 1\centerdot  & 4<2^{3}-1
& \text{\cite{P2}} \\
2^{4} & 9\centerdot  & 9\centerdot  & 9\centerdot  & 1\centerdot  & 7<2^{4}-1
& \text{\cite{P2}} \\
2^{5} & 14 & 14 & 14 & 1 & 12<2^{5}-1 & \text{\cite{P2}} \\
2^{6} & 22 & 22 & 22 & 1 & 20<2^{6}-1 & \text{\cite{P2}} \\
2^{7} & 32 & 34 & 34 & 1 & 32<2^{7}-1 & \text{\cite[Sec.\thinspace 5,
Lem.\thinspace 4.3]{Gi}} \\
2^{8} & 55 & 55 & 55 & 2 & 106<2^{8}-1 & \star, \text{ \cite{FainaJiul}}\\
2^{9} & 86 & 86 & 86 & 3 & 252<2^{9}-1 & \text{\cite{P2}} \\
2^{10} & 124 & 124 & 124 & 1 & 122<2^{10}-1 &\star, \text{ \cite{FainaJiul}},
\text{ Lemma \ref{Lem_Kw'}} \\
2^{11} & 201 & 201 & 201 & 4 & 796<2^{11}-1 & \star  \\
2^{12} & 307 & 307 & 307 & 5 & 1525<2^{12}-1 & \star, \text{ \cite{FainaJiul}} \\
2^{13} & 461 & 461 & 461 & 6 & 2754<2^{13}-1 & \star\\
2^{14} & 665 & 665 & 665 &  11&7293 <2^{14}-1 & \star, \text{ \cite{FainaJiul}},
\text{ Lemma \ref{Lem_Kw'}} \\
2^{15} & 2^{10} & 2^{10}+2 & 2^{10}+2 & 1 & 2^{10}<2^{15}-1 & \text{\cite[
Sec.\thinspace 5, Lem.\thinspace 4.3]{Gi}} \\
2^{16} & \frac{1}{3}(2^{16}+8) & \frac{1}{3}(2^{16}+8) & \frac{1}{3}%
(2^{16}+8) & 1 & \frac{1}{3}(2^{16}+2)<2^{16}-1 & \text{\cite{ABA}, Lemma
\ref{Lem_Abatangelo}} \\
2^{17} & 2^{16}\leq  & 2^{16}+2\leq  & 2^{16}+2\leq  & 1 & 2^{16}<2^{17}-1 &
\text{\cite{Gi}, Lemma \ref{Lem_q/2+2}} \\
2^{18} & 3066 & 3066 & 2^{17}+2 & 1 & 2^{17}<2^{18}-1 & \text{\cite{DGMP},\cite{Gi},
Lemma \ref{Lem_q/2+2}} \\ \hline
\end{array}
\end{equation*}

\vspace*{0.3cm}

On the basis of Table 1, together with the results of some computer
search, we make the following conjecture.
\begin{conjecture}
\label{Conj} Every complete arc in $PG(2,q)$ is projectively
equivalent to an arc with only one  sum-point.
\end{conjecture}

Let $K$ be an affinely complete arc in $PG(2,q)$. Without loss
of generality, assume that the point $(0,0,1)$ is covered by the
secants of $K$. The following results will be needed in the sequel.

\begin{lemma}\label{primo}
For an affinely complete arc $K$ in $PG(2,q)$ such that $(0,0,1)$ is
covered by the secants of $K$, it  can be assumed that $1\notin
S_\infty(K)$.
\end{lemma}
\begin{proof} See the proof of Lemma \ref{Lem_Sinfnot1}.
\end{proof}

\begin{lemma}\label{secondo} Let $K$  be an affinely complete $k$-arc $K$ in $PG(2,q)$ such that $(0,0,1)$ is
covered by the secants of $K$. Assume that $k<q-5$. Then there exist
$m_1,m_2\in \fq^*$ with $m_1\ne m_2$, $(m_1+m_2)^3\ne 1$, $m_i\ne
1$, such that $$1\notin S_{m_1}(K)\cup S_{m_2}(K).$$
\end{lemma}
\begin{proof}
First we prove that the number of values of $m$ for which $1\in
S_m(K)$ is at most $k$. For $a \in \fq$, let $n_a$ be the number of
pairs $(1,a,b)$, $(1,a+1,d)$ of points in $K$. Let $l_1$ be the line
of equation $X_1=aX_0$, and $l_2$ that of equation $X_1=(a+1)X_0$.
It is straightforward to check that $n_a$ is equal to
\begin{itemize}
\item $4$, if both lines $l_1$ and $l_2$ are secants to $K$; \item
$2$, if one of the two lines is a 1-secant to $K$, and the other is
a 2-secant; \item $1$, if both lines are $1$-secant to $K$; \item
$0$, if at least one line is a $0$-secant to $K$.
\end{itemize}
Then $n_a$ is less than or equal to  the number of points in $K$
belonging to $l_1\cup l_2$. When $a$ ranges over $\fq$,  we obtain
at most $k$ pairs
 $(1,a,b)$, $(1,a+1,d)$ of points in $K$.
Then by the assumption $k<q-5$, there are at
 least $6$ values of $m\in \fq$ such that $1\notin S_m(K)$.
Let $\theta_1$ and $\theta_2$ be the roots of $T^3=1$ distinct from
$1$. Choose $m_1\notin \{0,1\}$ such that $1\notin S_{m_1}(K)$. Then
there exists $m_2\notin \{0,1,m_1,\theta_1+ m_1,\theta_2+m_1\}$,
with $1\notin S_{m_2}(K)$. This completes the proof.
\end{proof}

\section{Caps in projective spaces of even dimension}\label{main}

Throughout this section the following notation is fixed. Let $s$ be
a positive integer. Let $X_{0},X_{1},\ldots ,X_{2s+2}$ be
homogeneous coordinates for the points of $PG(2s+2,q)$. For
$i=0,\ldots, 2s+1$, let $H_i$ be the subspace of $PG(2s+2,q)$ of
equations $X_0=\ldots=X_{i}=0$. Let $AG(N,q)$ be the $N$-dimensional
affine space over $\fq$. As usual, a point in $AG(N,q)$ is
identified with a vector in $\fq^N$. For any integer $j\geq 1$, let
\begin{equation}\label{cart}
\cP^j=\{(a_1,a_1^2,\ldots,a_{j},a_{j}^2)\mid a_1,\ldots,a_{j}\in
\fq\}\,\subset AG(2j,q).
\end{equation}
The set $\cP^j$ is a cap in $AG(2j,q)$, as it was first noticed in
\cite{PS}.

The so called {\em product construction}, first introduced in
\cite{MUK}, is the starting point for our constructions of small
complete caps in $PG(2s+2,q)$.
\begin{proposition}[see \cite{EDE}]
Let $C_1\subset \fq^{N_1+1}$ be a set of representatives of a cap
$C=\langle C_1 \rangle\subset PG(N_1,q)$,  and  let $C_2\subset
AG(N_2,q)$ be a cap. Then the product
$$
(C:C_2):=\{(P,Q)\mid P\in C_1, Q\in C_2\}\subset
PG(N_1+N_2,q)\,\,\text{ is a cap.}
$$
\end{proposition}

In this section we consider products $(K:\cP^s)$, where $K$ is an
arc in $PG(2,q)$. Completeness of $K$ in $PG(2,q)$ is not enough to
guarantee the completeness of $(K:\cP^s)$ in $PG(2s+2,q)$. In order
to obtain a complete cap, the following inductive construction of a
cap in $PG(2s+1,q)$ will be a key tool.

 Let $m_1,m_2\in \fq^*$ with
$m_1\ne m_2$, $(m_1+m_2)^3\ne 1$, $m_i\ne 1$. Let
$K_{m_1,m_2}^{(1)}$ be the subset of $PG(1,q)$ consisting of points
$\{(1,0),(0,1)\}$. For $i\ge 1$, let
\begin{equation}\label{m1m2}
K_{m_1,m_2}^{(2i+1)}=A_1^{(2i+1)}\cup A_2^{(2i+1)}\cup
A_3^{(2i+1)}\cup \{(1,0,\ldots,0)\}\subset PG(2i+1,q)
\end{equation}
where
$$
\begin{array}{ll}
A_1^{(2i+1)}=&\{(1,m_1,a_1,a_1^2,a_2,a_2^2,\ldots,a_i,a_i^2),\\
{} &
(1,m_2,a_1,a_1^2,a_2,a_2^2,\ldots,a_i,a_i^2)\mid\\
{} & a_1,\ldots,a_i\in\fq, (a_1,\ldots,a_i)\ne(0,\ldots,0) \},
\end{array}
$$

$$
\begin{array}{ll}
A_2^{(2i+1)}=&\{(0,1,a_1,a_1^2,a_2,a_2^2,\ldots,a_i,a_i^2)\mid\\
{} & a_1,\ldots,a_i\in\fq, (a_1,\ldots,a_i)\ne(0,\ldots,0) \},
\end{array}
$$

$$
A_3^{(2i+1)}=\{(0,0,b_0,b_1,\ldots,b_{2i-1})\mid
(b_0,b_1,\ldots,b_{2i-1})\in K^{(2i-1)}_{m_1,m_2}\}.
$$

Let $K_2^*(m_1,m_2)=K^{(2s+1)}_{m_1,m_2}\setminus
\{(1,0,\ldots,0)\}$.
\begin{proposition}\label{chiave}
If $q>4$, then the set $K_2^*(m_1,m_2)$ is a cap in $PG(2s+1,q)$
which covers all the points in $PG(2s+1,q)$ with the exception of
points
$$
(1,m,0,0\ldots,0),\,\,\,m\in \fq.
$$
\end{proposition}
Being quite technical, the proof of Proposition \ref{chiave} is
postponed in the Appendix.

We are now in a position to construct complete caps from products
$(K:\cP^s)$, $K$ being a suitable arc in $PG(2,q)$. Three cases will
be investigated.
\begin{itemize}
\item[{\rm (I)}] $K$ affinely complete.
\item[{\rm (II)}] $K$ complete, $\beta(K)=1$.
\item[{\rm (III)}] $K$ complete, $\beta(K)=1$, $(k-2)p(K)<q-1$.
\end{itemize}

\subsection{Case (I)} By Lemma \ref{primo} we can assume without loss of generality that
\begin{itemize}
\item[{\rm (Ia)}] $K$ is affinely complete, $(0,0,1)$ is covered
by $K$, $1\notin S_\infty(K)$.
\end{itemize}

\begin{lemma}\label{lem0} The cap $(K:\cP^s)$ covers all points in ${PG(2s+2,q)\setminus H_0}$.
\end{lemma}
\begin{proof} Let $Q=(1,\alpha,\beta,c_1,c_1',\ldots,c_s,c_{s}')$ be any
point in $PG(2s+2,q)\setminus H_0$. Write
$$
(1,\alpha,\beta)=\gamma(1,a,b)+(\gamma+1) (1,c,d)
$$
with $(1,a,b)$, $(1,c,d)$ in $K$.  Assume first that $\gamma\notin
\{0,1\}$. We look for $\lambda_i$, $\mu_i$ in $\fq$ such that
$$
Q=\gamma(1,a,b,\lambda_1,\lambda_1^2,\ldots,\lambda_s,\lambda_s^2)+(\gamma+1)(
1,c,d,\mu_1,\mu_1^2,\ldots,\mu_s,\mu_s^2),
$$
that is,
$$
\left\{\begin{array}{l} c_i^2=
\gamma^2 \lambda_i^2+(\gamma+1)^2 \mu_i^2\\
c_i'= \gamma \lambda_i^2+(\gamma+1) \mu_i^2
\end{array}\right. \,\,\,\,\,\,\text{ for each }i=1,\ldots,s.
$$
As $\gamma(\gamma+1)\ne 0$, such $\lambda_i,\mu_i$ certainly exist.

We now need to consider the case $\gamma\in \{0,1\}$, that is,
$(1,\alpha,\beta)\in K$.  Fix any $\delta\notin \{0,1\}$. We look
for $\lambda_i$, $\mu_i$ in $\fq$ such that
$$
Q=\delta(1,\alpha,\beta,\lambda_1,\lambda_1^2,\ldots,\lambda_s,\lambda_s^2)+(\delta+1)(
1,\alpha,\beta,\mu_1,\mu_1^2,\ldots,\mu_s,\mu_s^2),
$$
that is,
$$
\left\{\begin{array}{l} c_i^2=
\delta^2 \lambda_i^2+(\delta^2+1) \mu_i^2\\
c_i'= \delta \lambda_i^2+(\delta+1) \mu_i^2
\end{array}\right.\,\,\,\,\,\text{ for each }i=1,\ldots,s.
$$
As $\delta^2(\delta+1)+(\delta^2+1)\delta=\delta(\delta+1)\ne 0$,
such $\lambda_i,\mu_i$ certainly exist.
\end{proof}

\begin{lemma}\label{lem3} The points on $H_0$ covered by $(K:\cP^s)$ are
\begin{eqnarray*}
(0,1,m, a_1,A  a_1^2,\ldots, a_s,A  a_s^2 ),\,\,m\in
Cov_\infty(K),\,\, A \in S_m(K),\,\, a_i\in \fq, \\
(0,0,1, a_1,A a_1^2,\ldots, a_s,A  a_s^2),\,\, A \in
S_\infty(K),\,\,
a_i\in \fq, \\
(0,0,0,\ldots,1,m, a_{l}, m a_{l}^2,\ldots, a_s,m a_s^2), \,\,l\ge
2,\,\,m\in \fq^*,\,\ a_i\in \fq,
\\
(0,0,\ldots,0,1,m), \,\,m\in \fq^*.
\end{eqnarray*}
\end{lemma}
\begin{proof}
Let $Q_1=(1,a,b,\lambda_1,\lambda_1^2,\ldots,\lambda_s,\lambda_s^2)$
and $Q_2=(1,c,d,\mu_1,\mu_1^2,\ldots,\mu_s,\mu_s^2)$ be two points
in $(K:\cP^s)$. The line through $Q_1$ and $Q_2$ meets $H_0$ in
$$
Q=(0,a+c,b+d,(\lambda_1+\mu_1),(\lambda_1+\mu_1)^2,\ldots,(\lambda_s+\mu_s),(\lambda_s+\mu_s)^2).
$$
Assume that $a+c\ne 0$. Let $m=\frac{b+d}{a+c}$, $
a_i=(\lambda_i+\mu_i)/(a+c)$. Then
$$
Q=(0,1,m, a_1,(a+c) a_1^2,\ldots, a_s,(a+c) a_s^2).
$$
If $a+c= 0$, $b+d\ne 0$, let $ a_i=(\lambda_i+\mu_i)/(b+d)$. Then
$$
Q=(0,0,1, a_1,(b+d) a_1^2,\ldots, a_s,(b+d) a_s^2).
$$
Finally, assume that $a+b=0,c+d=0$. Let $l$ be the minimum integer
for which $\lambda_l+\mu_l\ne 0$, and let $m=\lambda_l+\mu_l$. If
$l<s$, let $ a_i=(\lambda_i+\mu_i)/(\lambda_l+\mu_l)$. Then
$$
Q=(0,0,0,\ldots,1,m, a_{l+1}, m a_{l+1}^2,\ldots, a_s,m  a_s^2).
$$
If $l=s$, then $Q=(0,0,\ldots,0,1,m)$.
\end{proof}

 Let
$m_1$, $m_2$  be as in Lemma \ref{secondo}. Let $K_2(m_1,m_2)$ be
the natural embedding of $K_2^*(m_1,m_2)$ in $H_0\subset
PG(2s+2,q)$.

\begin{proposition}\label{PROP0} Assume that  $k<q-5$.
Then $$\cK=(K:\cP^s)\cup K_2(m_1,m_2)\,\,\,\text{ is a cap }.$$
\end{proposition}
\begin{proof} Note that
$ 1\notin (S_{m_1}(K)\cup S_{m_2}(K)\cup S_{\infty}(K)) $. Then by
Lemma \ref{lem3}  no point in $K_2(m_1,m_2)$ is covered by
$(K:\cP^s)$; the converse is also true as $K_2(m_1,m_2)\subset H_0$
and $(K:\cP^s)\cap H_0=\emptyset$. Then no three points in
$(K:\cP^s)\cup K_2(m_1,m_2)$ are collinear.
\end{proof}
\begin{theorem}\label{TEO1} Let $M=2s+2$, $s\ge 1$, $q>8$. Assume that {\rm{(Ia)}} holds and $k<q-5$.
Let $\cK$ be the cap $$\cK=(K:\cP^s)\cup K_2(m_1,m_2)\subset
PG(M,q).$$ Then the size of $\cK$ is $$(k+3)\cdot q^\frac{M-2}{2}
+3(q^\frac{M-4}{2}+q^\frac{M-6}{2}+\ldots+q)-M+3.$$ Moreover,
\begin{itemize}
\item if $Cov_\infty(K)=\fq$, then $\cK$ is a complete cap;

\item if $\fq\setminus Cov_\infty(K)=\{m_0\}$, then
$$
\cK'=\cK\cup \{(0,1,m_0,0,\ldots, 0)\}
$$
is a complete cap;

\item if $\fq\setminus Cov_\infty(K)\supseteq \{m_0,m_0'\}$, then
$$
\cK'=\cK\cup \{(0,1,m_0,0,\ldots, 0),(0,1,m_0',0,\ldots, 0)\}
$$
is a complete cap.
\end{itemize}

\end{theorem}
\begin{proof} The claim on the size of $\cK$ follows from
straightforward computation. The points in $PG(M,q)\setminus H_0$
are covered by $(K:\cP^s)$ according to Lemma \ref{lem0}.  Points of
$H_0$ not of type $(0,1,m,0,\ldots,0)$ are covered by $K_2(m_1,m_2)$
according to Proposition \ref{chiave}. Points in $H_0$ of type
$(0,1,m,0,\ldots,0)$, $m\in Cov_\infty(K)$, are covered by
$(K:\cP^s)$ according to Lemma \ref{lem3}. Then the claim follows.
\end{proof}

\subsection{Case (II)} By Lemma \ref{Lem_Sinfnot1} we can assume without loss of generality that
\begin{itemize}
\item[{\rm (IIa)}] $K$ is complete, $\beta(K)=1$, the only
sum-point for $K$ is $(0,0,1)$, $1\notin S_\infty(K)$.
\end{itemize}

\begin{lemma}\label{lem0B} The cap $(K:\cP^s)$ covers all points in $PG(2s+2,q)\setminus H_1$.
\end{lemma}
\begin{proof}
Let $Q=(\delta,\alpha,\beta,c_1,c_1',\ldots,c_s,c_{s}')$ be any
point in $PG(2s+2,q)\setminus H_1$, where $(\delta,\alpha,\beta)$ is
written in its normalized form. As $(\delta,\alpha,\beta)\neq
(0,0,1)$ there exists $\gamma_1\neq \gamma_2$ such that
$$
(\delta,\alpha,\beta)=\gamma_1 (Y_0,Y_1,Y_2)+\gamma_2
(Y_0',Y_1',Y_2')
$$
with $(Y_0,Y_1,Y_2)$, $(Y_0',Y_1',Y_2')$ in $K$. The existence of
$\lambda_i$, $\mu_i$ in $\fq$ such that
$$
Q=\gamma_1(Y_0,Y_1,Y_2,\lambda_1,\lambda_1^2,\ldots,\lambda_s,\lambda_s^2)+\gamma_2(
Y_0',Y_1',Y_2',\mu_1,\mu_1^2,\ldots,\mu_s,\mu_s^2),
$$
can then be proved as in the proof of Lemma \ref{lem0}.
\end{proof}

\begin{lemma}\label{lem3B}
The points on $H_1$ covered by $(K:\cP^s)$ are
\begin{eqnarray}
(0,0,1,a _{1},Aa _{1}^{2},\ldots ,a _{s},Aa _{s}^{2}),\quad A
&\in & S_{\infty }(K),a _{i}\in \fq.  \label{form_coverK*M_top} \\
(0,0,0,\ldots ,1,m,a _{l},ma _{l}^{2},\ldots ,a _{s},ma
_{s}^{2}),\quad l &\geq &2,m\in \fq^{\ast },a _{i}\in \fq.
\label{form_coverK*M_bot} \\
(0,0,0\ldots ,,0,0,1,m),\quad m &\in &\fq^{\ast }.
\label{form_coverK*M_bot2}
\end{eqnarray}%
\end{lemma}
\begin{proof} The proof is similar to that of Lemma \ref{lem3}.
\end{proof}

Let $m_1$ and $m_2$ be as in Lemma \ref{secondo}. Let ${\bar
K}_2^*(m_1,m_2)=K_{m_1,m_2}^{(2s-1)}\setminus \{(1,0,\ldots,0)\}$,
and let  ${\bar K}_2(m_1,m_2)$ be the natural embedding of ${\bar
K}_2^*(m_1,m_2)$ in the subspace $H_2$ of $PG(2s+2,q)$ of equations
$X_0=X_1=X_2=0$.
\begin{proposition}\label{kappa0} Assume that  $k<q-5$. Let $m_1$, $m_2$  be as in Lemma \ref{secondo}. Then
the set
$$
\cK:=(K:\cP^s)\cup {\bar K}_2(m_1,m_2)\cup
\{(0,0,1,a_1,a_1^2,\ldots,a_s,a_s^2)\mid a_i\in \fq\}$$ is a cap.
\end{proposition}
\begin{proof} Let ${\bar K}^{(0)}$ denote the cap $\{(0,0,1,a_1,a_1^2,\ldots,a_s,a_s^2)\mid a_i\in
\fq\}$. By Lemma \ref{lem3B} $(K:\cP^s)\cup {\bar K}^{(0)}$ is a cap
since $1\notin S_\infty(K)$. Again by Lemma \ref{lem3B},
$(K:\cP^s)\cup {\bar K}_2(m_1,m_2)$ is a cap as well. Clearly a
point in ${\bar K}_2(m_1,m_2)$ cannot be collinear with a point in
$(K:\cP^s)$ and a point in ${\bar K}^{(0)}$. Then it remains to show
that no two points in ${\bar K}^{(0)}$ are collinear with a point in
${\bar K}_2(m_1,m_2)$. But this follows from the fact that points in
$H_2$ covered by ${\bar K}^{(0)}$ are points of type
$(0,0,0,\ldots,1,m, a_{l+1}, m a_{l+1}^2,\ldots, a_s,m  a_s^2)$.
\end{proof}

\begin{theorem}\label{TEO2}  Let $M=2s+2$, $s\ge 1$, $q>8$. Assume that {\rm (IIa)} holds and $k<q-5$.
Let $\cK$ be the cap
$$
\cK:=(K:\cP^s)\cup {\bar K}_2(m_1,m_2)\cup
\{(0,0,1,a_1,a_1^2,\ldots,a_s,a_s^2)\mid a_i\in \fq\}\subset
PG(M,q).
$$
Then,
\begin{itemize}
\item the size of $\cK$ is $$(k+1)\cdot q^\frac{M-2}{2}
+3(q^\frac{M-4}{2}+q^\frac{M-6}{2}+\ldots+q)-M+5;$$ \item  $\cK$
is  complete.
\end{itemize}
\end{theorem}
\begin{proof}
The claim on the size of $\cK$ follows by straightforward
computation. Let ${\bar K}^{(0)}$ be as in the proof of Proposition
\ref{kappa0}. The points in $PG(M,q)\setminus H_1$ are covered by
$(K:\cP^s)$ according to Lemma \ref{lem0B}. It is straightforward to
check that ${\bar K}^{(0)}$ covers all the points in $H_1\setminus
H_2$, with the exception of $(0,0,1,0,\ldots,0)$, which is covered
by $(K:\cP^s)$. Points of $H_2$ not of type $(0,0,0,1,m,0,\ldots,0)$
are covered by ${\bar K}_2(m_1,m_2)$ according to Proposition
\ref{chiave}. Points in $H_2$ of type $(0,0,0,1,m,0,\ldots,0)$,
 are covered by $(K:\cP^s)$ according to Lemma
\ref{lem3B}. Then the claim follows.
\end{proof}

\subsection{Case (III)} By Lemma \ref{Lem_forLaScala} we can assume without loss of generality that
\begin{itemize}
\item[{\rm (IIIa)}] $K$ is complete, $\beta(K)=1$, $K\cap
l_\infty=\{(0,0,1),(0,1,0)\}$,
$(k-2)p(K)<q-1$,\begin{equation}\label{caso3}K\cap \{(1,a,Aa^2)\mid
A \in S_1(K), a\in \fq\}=\emptyset.\end{equation}
\end{itemize}
We first consider the product cap $(K:\cP^j)$ in $PG(2j+2,q)$, with
$1\le j\le s$. Let $Y_0,\ldots,Y_{2j+2}$ be homogeneous coordinates
for points in $PG(2j+2,q)$.
\begin{lemma}\label{lem0C}
The cap $(K:\cP^j)$ covers all points in $PG(2j+2,q)$, but some
points on the subspace of equations $Y_0=0$, $Y_1=Y_2$.
\end{lemma}
\begin{proof} The proof is similar to that of Lemma \ref{lem0B}.
Note that here the only sum-point of $K$ is $(0,1,1)$.
\end{proof}

\begin{lemma}\label{lem3C}
The points on the subspace of $PG(2j+2,q)$ of equations $Y_0=0$,
$Y_1=Y_2$ covered by $(K:\cP^j)$ are
\begin{eqnarray*}
(0,1,1,a _{1},Aa _{1}^{2},\ldots ,a _{j-1},Aa _{j-1}^{2},a _{j},Aa
_{j}^{2}),\quad A &\in &S_{1}(K),a _{i}\in
F_{q}. \\
(0,0,0,\ldots,1,m,a _{l},ma _{l}^{2},\ldots ,a _{j},ma_{j}^{2}),\quad l &\geq &2,\,m\in F_{q}^{\ast },\,a _{i}\in F_{q}. \\
(0,0,0\ldots ,,0,0,1,m),\quad m &\in &F_{q}^{\ast }.
\end{eqnarray*}%
\end{lemma}
\begin{proof} The proof is similar to that of Lemma \ref{lem3}.
\end{proof}
For $j= 0,\ldots,s-1 $, let $V_{2j+2}$ be the $(2j+2)$-dimensional
subspace of $PG(2s+2,q)$ of equations
$X_0=\ldots=X_{2s-2j-2}=0,X_{2s-2j-1}=X_{2s-2j}$. Let $\Phi_j$ be
the following isomorphism between  $PG(2j+2,q)$ and  $V_{2j+2}$
$$
\Phi_j(Y_0,Y_1,\ldots,Y_{2j+2})=(0,0,\ldots,0,Y_0,Y_0,Y_1,\ldots,Y_{2j+2}).
$$
For $j=1,\ldots,s-1$, let $K^{(j)}$ be the image by $\Phi_j$ of
$(K:\cP^j)\subset PG(2j+2,q)$. Also, let $K^{(0)}$ be the image of
$K$ by $\Phi_0$.
%
\begin{proposition}\label{PROP2} The set
 $$ \cK:=(K:\cP^s)\,\,\bigcup \,\,\left(\bigcup_{j=0}^{s-1}K^{(j)}\right) \quad \text{is a cap.}$$
\end{proposition}
\begin{proof} Note that the subspace of
$PG(2j+2,q)$ of equations $Y_0=0$, $Y_1=Y_2$ is mapped by $\Phi_j$
onto $V_{2j}$ for any $j\ge 1$. Then  Lemma \ref{lem3C}, together
with (\ref{caso3}), yield that the subset of points covered by the
cap $K^{(j)}$ is disjoint from $\cup_{u=0}^{j-1}K^{(u)}$. This
proves the assertion.
\end{proof}

\begin{theorem}\label{TEO3}
Let $M=2s+2$, $s\ge 1$. Assume that {\rm (IIIa)} holds. Then the cap
$$ \cK:=(K:\cP^s)\,\,\bigcup \,\,\left(\bigcup_{j=0}^{s-1}K^{(j)}\right) \subset PG(M,q)$$
is a complete cap of size
$$
k\frac{q^{\frac{M}{2}}-1}{q-1}=k(q^{\frac{M-2}{2}}+q^{\frac{M-4}{2}%
}+q^{\frac{M-6}{2}}+\ldots +q+1).$$
\end{theorem}
\begin{proof}

The claim on the size of $\cK$ follows from straightforward
computation. Note that $(K:\cP^s)$ covers all the points in
$PG(M,q)\setminus V_{2s}$, and that for each $j=1,\ldots,s-1$, the
cap $K^{(j)}$ covers all the points in $V_{2j+2}\setminus V_{2j}$,
see Lemma \ref{lem3C}. Taking into account that $K^{(0)}$ covers
all the points in $V_{2}$,  the completeness of $\cK$ follows.
\end{proof}

By Theorem \ref{TEO3}, taking into account the values of
$\overline{t}_{2}^{\ast }(2,q)$ from Table 1, we obtain
complete $k_{4,q}$-caps in $PG(4,q)$ with the following sizes
(the best known sizes from \cite{Gi}, \cite{GiulPast},
\cite{PS}) are given in parentheses):
$k_{4,128}=34q+34$ $(35q+2),$ $%
k_{4,256}=55q+55$ $(67q+2),$ $k_{4,512}=86q+86$ $(131q+2)$,
$k_{4,1024}=124q+124$ $(131q+2)$, $k_{4,2048}=201q+201$
$(259q+2)$, $k_{4,4096}=307q+307$ $(515q+2)$. See also the
second table in Introduction.

\subsection{New upper bounds on $t_2(N,q)$ for $N$ and $q$ even}

\begin{theorem}\label{NEW1} Let $N$ and $q$ be even, $N>2$. If $q>8$, then
$$
t_2(N,q)\le (t_2(2,q)+3)\cdot q^\frac{N-2}{2}
+3(q^\frac{N-4}{2}+q^\frac{N-6}{2}+\ldots+q)-N+3,
$$
and
$$
t_2(N,q)\le (t_2^A(2,q)+3)\cdot q^\frac{N-2}{2}
+3(q^\frac{N-4}{2}+q^\frac{N-6}{2}+\ldots+q)-N+5,
$$
 where $t_2^A(2,q)\le t_2(2,q)$ is the size of the smallest
affinely complete arc in  $PG(2,q)$.
\end{theorem}
\begin{proof} The assertion follows from Theorem \ref{TEO1}.
\end{proof}

\begin{theorem}\label{NEW2} Let $N$ and $q$ be even, $N>2$. If $q>8$, then
$$
t_2(N,q)\le ({t}^S_2(2,q)+1)\cdot q^\frac{N-2}{2}
+3(q^\frac{N-4}{2}+q^\frac{N-6}{2}+\ldots+q)-N+5,
$$
where $t_2^S(2,q)$ is the size of the smallest complete arc in
$PG(2,q)$ with only one sum-point.
\end{theorem}
\begin{proof} The assertion follows from Theorem \ref{TEO2},
together with Lemma \ref{Lem_Sinfnot1}.
\end{proof}

\begin{theorem}\label{NEW3} Let $N$ and $q$ be even, $N>2$. Then
$$
t_2(N,q)\le {t}^{S^+}_2(2,q)(q^{\frac{N-2}{2}}+q^{\frac{N-4}{2}%
}+q^{\frac{N-6}{2}}+\ldots +q+1),
$$
where $t_2^{S^+}(2,q)$ is the size of the smallest complete
$k$-arc $K$ in $PG(2,q)$ with only one sum-point and with the
property that $(k-2)p(K)<q-1$.
\end{theorem}
\begin{proof} The assertion follows from Theorem \ref{TEO3},
together with Lemma \ref{Lem_forLaScala}.
\end{proof}

\section{Caps in projective spaces of odd dimension}\label{secODD}

We keep the notation of the previous section. We consider complete
$k$-arcs $K$ in $PG(2,q)$ such that:
\begin{itemize}
\item[(*)] $K$ is affine, $k<q-5$, $1\notin S_\infty(K)$, $Y_2\neq
Y_1^2$ for each point $(1,Y_1,Y_2)\in K$.
\end{itemize}
\begin{lemma}\label{odd1} Any complete $k$-arc in $PG(2,q)$ with $k<q-5$ is projectively equivalent to a
$k$-arc satisfying property {\rm{(*)}}.
\end{lemma}
\begin{proof} We can assume that $K$ is affine as every arc has an external line, which
can be moved to $l_\infty$. By Lemma \ref{primo} we can also assume
that $1\notin S_\infty(K)$. When $(1,a,b)$ ranges over $K$,  the number of values of $b/a^2$
is at most $k$. Therefore, there exists an element $w$ in $\fq$ such that
$b\neq (wa)^2$ for every point $(1,a,b)\in K$. The projectivity
$\psi(X_0,X_1,X_2)=(X_0,wX_1,X_2)$ then maps $K$ onto an arc
$\psi(K)$ such that $Y_2\neq Y_1^2$ for each point $(1,Y_1,Y_2)\in
\psi(K)$. It is straightforward that $S_\infty(\psi(K))$ coincides
with $S_\infty(K)$, whence the lemma is proved.
\end{proof}

Let $K_0=\{(1,1),(1,0)\}$ be the trivial complete cap in $PG(1,q)$.
We consider the product cap $(K_0:\cP^{s+1})\subset PG(2s+3,q)$. Let
$H_0$ be the subspace of $PG(2s+3,q)$ of equation $X_0=0$.
\begin{lemma}\label{odd2} The cap $(K_0:\cP^{s+1})\subset PG(2s+3,q)$ covers all
the points in $PG(2s+3,q)\setminus H_0$. Points in $H_0$ covered by
$(K_0:\cP^{s+1})$ are precisely
\begin{eqnarray*}
(0,1,a _{1},a _{1}^{2},\ldots ,a _{s+1},a _{s+1}^{2}),\quad a
_{i} &\in &\fq.  \label{form_23} \\
(0,0,\ldots ,1,m,a _{l},ma _{l}^{2},\ldots ,a
_{s+1},ma_{s+1}^{2}),\quad l &\geq &2,m\in \fq^{\ast },a _{i}\in
\fq.
\label{form_24} \\
(0,0\ldots ,0,0,1,m),\quad m &\in &\fq^{\ast }.  \label{form_25}
\end{eqnarray*}%
\end{lemma}
\begin{proof}
Let $Q=(1,\gamma,c_1,c_1',\ldots,c_s,c_{s}',c_{s+1},c_{s+1}')$ be
any point in $PG(2s+3,q)\setminus H_0$. Clearly
$$
(1,\gamma)=\gamma (1,1)+(1+\gamma) (1,0).
$$
By an argument analogous to that of the proof of Lemma \ref{lem0},
it can then be proved that  when $\gamma\notin \{0,1\}$ there exist
$\lambda_i$, $\mu_i$ in $\fq$ such that
$$
Q=\gamma(1,1,\lambda_1,\lambda_1^2,\ldots,\lambda_{s+1},\lambda_{s+1}^2)+(1+\gamma)(
1,0,\mu_1,\mu_1^2,\ldots,\mu_{s+1},\mu_{s+1}^2),
$$
 and that when $\gamma\in \{0,1\}$ there exist
$\delta,\lambda_i$, $\mu_i$ in $\fq$ such that
$$
Q=\delta(1,\gamma,\lambda_1,\lambda_1^2,\ldots,\lambda_{s+1},\lambda_{s+1}^2)+(1+\delta)(
1,\gamma,\mu_1,\mu_1^2,\ldots,\mu_{s+1},\mu_{s+1}^2).
$$
This proves that $(K_0:\cP^{s+1})$ covers all the points in
$PG(2s+3,q)\setminus H_0$. The proof of the second statement of the
Lemma is analogous to that of Lemma \ref{lem3}.
\end{proof}

Now, let $\cK\subset PG(2s+2,q)$ be as in Proposition \ref{PROP0}.
Let ${\bar \cK}$ be the natural embedding of $\cK$ in the hyperplane
$H_0$ of $PG(2s+3,q)$.

\begin{proposition}\label{PROPODD1} The set
$ (K_0:\cP^{s+1})\cup {\bar \cK} $ is a cap in $PG(2s+3,q)$.
\end{proposition}
\begin{proof} By Lemma \ref{odd2} together with property (*) it follows that ${\bar \cK} $ is disjoint from the
set of points in $H_0$ covered by $(K_0:\cP^{s+1})$.  This proves
the assertion.
\end{proof}

\begin{theorem}\label{TEODD1} Let $M=2s+3$, $s\ge 1$, $q>8$. Assume that  {\rm{(*)}} holds. Then the set
$ (K_0:\cP^{s+1})\cup {\bar \cK} $ is a complete cap in $PG(M,q)$
of size $$2q^{\frac{M-1}{2}}+(k+3)\cdot q^\frac{M-3}{2}
+3(q^\frac{M-5}{2}+q^\frac{M-7}{2}+\ldots+q)-M+4.$$
\end{theorem}
\begin{proof} By Proposition \ref{PROPODD1} the set $ (K_0:\cP^{s+1})\cup {\bar \cK} $ is a cap, which is complete by Lemma \ref{odd2}
and Theorem \ref{TEO1}.
\end{proof}

{}From now we assume that $K$ is a $k$-arc in $PG(2,q)$ satisfying
property (IIIa) of the previous section, and that $\cK\subset
PG(2s+2,q)$ is as in Proposition \ref{PROP2}. Let ${\bar \cK'}$ be
the natural embedding of $\cK$ in the hyperplane $H_0$ of
$PG(2s+3,q)$.
\begin{proposition}\label{PROPODD2} The set
$ (K_0:\cP^{s+1})\cup {\bar \cK'} $ is a cap in $PG(2s+3,q)$.
\end{proposition}
\begin{proof} Note that the plane arc $K$ is disjoint from points $\{(1,a,a^2)\mid a \in \fq\}$. Then
from Lemma \ref{odd2} it follows that the cap ${\bar \cK'} $ is
disjoint from the set of points in $H_0$ covered by
$(K_0:\cP^{s+1})$.  This proves the assertion.
\end{proof}

\begin{theorem}\label{TEODD2} Let $M=2s+3$, $s\ge 1$. Assume that \rm{(IIIa)} holds. Then the set
$ (K_0:\cP^{s+1})\cup {\bar \cK'} $ is a complete cap in
$PG(2s+3,q)$ of size
$$
2q^{\frac{M-1}{2}}+k(q^{\frac{M-3}{2}}+q^{\frac{M-5}{2}%
}+q^{\frac{M-7}{2}}+\ldots +q+1).
$$
\end{theorem}
\begin{proof}  The claim follows from Lemma \ref{odd2} together with
Theorem \ref{TEO3}.
\end{proof}

By Theorem \ref{TEODD2}, taking into account
 the values of $\overline{t}%
_{2}^{\ast }(2,q)$ from Table 1,  we obtain complete
$k_{5,q}$-caps in $ PG(5,q)$ with the following sizes (the best
known sizes from \cite{Gi,GiulPast,PS} are given in
parentheses):
 $k_{5,128}=2q^{2}+34q+34$ $
(2q^{2}+35q+2),$ $k_{5,256}=2q^{2}+55q+55$ $(2q^{2}+67q+2),$ $
k_{5,512}=2q^{2}+86q+86$ $(2q^{2}+131q+2)$,
$k_{4,1024}=2q^{2}+124q+124$ $(2q^{2}+131q+2)$,
$k_{4,2048}=2q^{2}+201q+201$ $(2q^{2}+259q+2)$,
$k_{4,4096}=2q^{2}+307q+307$ $(2q^{2}+515q+2)$.  See also the
second table in Introduction.

\subsection{New upper bounds on $t_2(N,q)$ for $N$ odd, $q$ even}

\begin{theorem}\label{NEW4} Let $N$ be odd, $N>3$, and let $q$ be even. If $q>8$, then
$$
t_2(N,q)\le 2q^{\frac{N-1}{2}}+(t_2(2,q)+3)\cdot q^\frac{N-3}{2}
+3(q^\frac{N-5}{2}+q^\frac{N-7}{2}+\ldots+q)-N+4.
$$
\end{theorem}
\begin{proof} The assertion follows from Theorem \ref{TEODD1} together with Lemma \ref{odd1}.
\end{proof}

\begin{theorem}\label{NEW5} Let $N$ be odd, $N>3$, and let $q$ be even. Then
$$
t_2(N,q)\le
2q^{\frac{N-1}{2}}+t_2^{S^+}(2,q)(q^{\frac{N-3}{2}}+q^{\frac{N-5}{2}%
}+q^{\frac{N-7}{2}}+\ldots +q+1),
$$
where $t_2^{S^+}(2,q)$ is the size of the smallest complete
$k$-arc $K$ in $PG(2,q)$ with only one sum-point and with the
property that $(k-2)p(K)<q-1$.
\end{theorem}
\begin{proof} The assertion follows from Theorem \ref{TEODD2},
together with Lemma \ref{Lem_forLaScala}.
\end{proof}

\section{Appendix: Proof of Proposition \ref{chiave}}\label{tecn}
To prove Proposition \ref{chiave}, some ideas and results on
translation caps from \cite{Gi} will be useful. A {\em translation
cap} in an affine space $AG(M,q)$ is a cap corresponding to a
coset of an additive subgroup of $\fq^M$. The prototype of a
translation cap in $AG(2,q)$ is the parabola $\cP=\{(a,a^2)\mid a
\in \fq\}$. The cartesian product of translation caps is still a
translation cap, see \cite[Lemma 2.7]{Gi}, whence the cap $\cP^i$
 defined as in (\ref{cart}) is a translation cap.

Let $m_1,m_2\in \fq^*$, with $m_1\ne m_2$, $(m_1+m_2)^3\ne 1$,
$m_i\ne 1$. Let $K_{m_1,m_2}^{(2s+1)}$, $A_1^{(2s+1)}$,
$A_2^{(2s+1)}$, $A_3^{(2s+1)}$ be as in Section \ref{main}.

Throughout this section,  let $(X_1,X_2,\ldots,X_{2s+2})$ denote
homogenous coordinates for the points in $PG(2s+1,q)$. Also, let
$L_1$ be the hyperplane of equation $X_1=0$, and $L_2$ be the
$(2s-1)$-dimensional subspace of equations $X_1=X_2=0$. Denote
$U=(1,0,\ldots ,0)$.

\begin{lemma}[{\cite[Proposition 2.5]{Gi}}]\label{GI1} Let $q>2$. Then through every point
in  $AG(2i,q)\setminus \cP^i$ there pass exactly $\frac{q-2}{2}$
secants of $\cP^i$.
\end{lemma}
\begin{lemma}\label{GI2}  Let $h_1,h_2$ be any distinct
elements in $\fq$. Let $C^{(i)}_{h_1,h_2}\subset AG(2i+1,q)$ be
the cartesian product of $\{h_1,h_2\}$ by $\cP^i$.  Then
$C^{(i)}_{h_1,h_2}$ is a cap. Moreover, through every point in
$AG(2i+1,q)\setminus C^{(i)}_{h_1,h_2}$ there pass at least one
secant of $C^{(i)}_{h_1,h_2}$.
\end{lemma}
\begin{proof} By  \cite[Proposition 2.9]{Gi} the assertion holds for
$h_1=0$, $h_2=1$. As clearly any $C^{(i)}_{h_1,h_2}$ is affinely
equivalent to $C^{(i)}_{0,1}$, the claim follows.
\end{proof}

The set $PG(2s+1,q)\setminus L_1$ can be viewed as affine space
$AG(2s+1,q)$, and similarly  $L_1\setminus L_2$ as an affine space
$AG(2s,q)$. Note that for $s>0$,
\begin{equation}\label{link1}
\begin{array}{c}
A_1^{(2s+1)}\cup \{(1,m_1,0\ldots,0),(1,m_2,0\ldots,0)\}=
C^{(s)}_{m_1,m_2},\\ A_2^{(2s+1)}\cup \{(0,1,0\ldots,0)\}= \cP^s.
\end{array}
\end{equation}

We are now in a position to prove that $ K_{m_1,m_2}^{(2s+1)} $ is
a cap.
\begin{lemma} The set $ K_{m_1,m_2}^{(2s+1)} $ is a cap.
\end{lemma}
\begin{proof}
We prove the claim by induction on $s$. The case $s=0$ is trivial.
Assume that $s>0$. Throughout this proof, let
$K=K_{m_1,m_2}^{(2s+1)}$ and $A_j=A_j^{(2s+1)}$ for $j=1,2,3$.

Note that $A_3$ is a cap by inductive hypothesis. Lemma \ref{GI2}
together with (\ref{link1}) yields that $A_1$ is a cap. Similarly,
$A_2$ is a cap by (\ref{link1}).

Assume  that $K$ is not a cap. Let $P_1,P_2$ and $P_3$ be
collinear points in $K$. Assume first that no $P_i$ coincides with
$U$.  Let
$$
P_1=(X_1^{(1)},X_2^{(1)},\ldots,X_{2s+2}^{(1)}),\,
P_2=(X_1^{(2)},X_2^{(2)},\ldots,X_{2s+2}^{(2)}),\,
P_3=(X_1^{(3)},X_2^{(3)},\ldots,X_{2s+2}^{(3)}).
$$
 Let $v_j=\min\{v\mid X_v^{(j)}\neq 0\}$. Assume without loss of
 generality that $v_1\le v_2 \le v_3$. Note that it is impossible
 that $v_1<v_2$, as in this case the line through $P_2$ and $P_3$
 is contained in the subspace $T:X_1=X_2=\ldots=X_{v_2-1}=0$,
 whereas $P_1\notin T$.
Then $v_1=v_2\le v_3. $ Note also that  $v_1=v_2=v_3$ cannot
occur, as in this case $\{P_1,P_2,P_3\}\subset A_j$ for some
$j=1,2,3$. Moreover, as it is not possible  that
$\{P_1,P_2,P_3\}\subset A_3$, we have that $v_1\le 2$ holds. Then
we are left with the following two cases.

{\em Case (1)}: {\em $v_1=v_2=1, v_3\ge 2$}.  Write $P_3=\lambda
P_1+\mu P_2$, $P_1=(1,m_k,a_1,a_1^2,\ldots,a_s,a_s^2)$,
$P_2=(1,m_l,b_1,b_1^2,\ldots,b_s,b_s^2)$, with $k,l\in \{1,2\}$.
Then clearly $\lambda=\mu$. Note that
$$
P_3=\lambda(P_1+P_2)=(0,\lambda(m_k+m_l),\lambda c_1,\lambda
c_1^2,\ldots,\lambda c_s,\lambda c_s^2),\,
$$
where $c_u=a_u+b_u$, $u=1,\ldots,s$.

Assume that $m_k=m_l$. Note that $c_u\ne 0$ for some $u$, otherwise
$P_1=P_2$. Let $v$ be the minimum $u$ for which $c_u\ne 0$. Then
 $v<s$, otherwise
$$P_3=\lambda(P_1+P_2)=(0,\ldots,0,\lambda c_s,\lambda c_s^2)=(0,\ldots,0,1,c_s)\notin K\,.$$
 Therefore $P_3=(0,\ldots,0,1,m_k,d_{v+1},d_{v+1}^2,\ldots,d_s,d_s^2)$
 for some $k\in \{1,2\}$, $d_{v+1},\ldots,d_s\in \fq$, whence
\begin{equation*}
(1,m_k,d_{v+1},d_{v+1}^2,\ldots,d_s,d_s^2)=(\lambda c_{v},\lambda
c_{v}^2,\lambda c_{v+1},\lambda c_{v+1}^2,\ldots,\lambda
c_s,\lambda c_s^2)\,
\end{equation*}
holds. This implies $\lambda =1/c_{v},$ $m_{k}=c_{v}$, and therefore $d_{s}=c_{s}/m_{k}$, $%
d_{s}^{2}=c_{s}^{2}/m_{k}$, which is impossible as $m_{k}\neq 1$.
%
%

Assume then that $m_k\ne m_l$. Then $P_3\in A_2$, and hence
$P_3=(0,1,d_1,d_1^2,\ldots,d_s,d_s^2)$
 for some  $d_{1},\ldots,d_s\in \fq$. Then
\begin{equation*}
(1,d_1,d_1^2,\ldots,d_s,d_s^2) =(\lambda(m_1+m_2),\lambda
c_1,\lambda c_1^2,\ldots,\lambda c_s,\lambda c_s^2)\,
\end{equation*}
holds.  As $P_3\in K$, $d_u\neq 0$ for some $u$. Then
$d_u=\frac{c_u}{m_1+m_2}$, $d_u^2=\frac{c_u^2}{m_1+m_2}$, which
yields $(m_1+m_2)^2=m_1+m_2$. But this is impossible as
$m_1+m_2\notin \{0,1\}$.

{\em Case (2)}: {\em $v_1=v_2=2$, $v_3\ge 3$}. Write $P_3=\lambda
P_1+\mu P_2$, with $P_1=(0,1,a_1,a_1^2,\ldots,a_s,a_s^2)$,
$P_2=(0,1,b_1,b_1^2,\ldots,b_s,b_s^2)$. Then clearly
$\lambda=\mu$. Note that
$$
P_3=\lambda(P_1+P_2)=(0,0,\lambda c_1,\lambda c_1^2,\ldots,\lambda
c_s,\lambda c_s^2),\,
$$
where $c_l=a_l+b_l$.  Then a contradiction can be obtained as in
case (1), $m_k=m_l$.

To complete the proof we only need to show that the point $U$ is not
collinear with two points $P_2,P_3$ in $K\setminus\{U\}$. Clearly,
either $P_2$ or $P_3$  belongs to $A_1$. Assume without loss of
generality that $P_2=(1,m_k,a_1,a_1^2,\ldots,a_s,a_s^2)\in A_1$. We
first deal with the case $P_3\in A_1$. Let
$P_3=(1,m_l,b_1,b_1^2,\ldots,b_s,b_s^2)$. Write $U=\lambda P_2+\mu
P_3$. If $a_u=0$ and $b_u\ne 0$ for some $u$, then $\mu=0$, which is
impossible as $U\neq P_2$. Similarly the case $a_u\ne 0$ and $b_u=
0$ can be ruled out. By definition $a_{u}\neq 0$ for some $u$.
Therefore $b_u\neq 0$. Note that
$$
D\left(\begin{array}{ccc} 1 & 0 & 0\\ 1 & a_u & a_u^2\\
1 & b_u &b_u^2
\end{array}\right)=0,
$$
yields $a_u=b_u$. Then, from
$$
D\left(\begin{array}{ccc} 1 & 0 & 0 \\ 1 & m_k & a_u\\
1 & m_l &b_u
\end{array}\right)=0
$$
it follows $m_k=m_l$. Then $P_2=P_3$ follows, a contradiction.
Assume then that $P_3\in A_2\cup A_3$. In this case
$$P_3=\lambda(U+P_2)=(0,m_k,a_1,a_1^2,\ldots,a_s,a_s^2).$$
Therefore $P_3\in A_2$, and hence
$$
P_3=(0,1,b_1,b_1^2,\ldots,b_s,b_s^2)
$$
for some $b_u\in \fq$. Then, for any $u$ with $b_u\ne 0$, we have
$ a_u/m_k=b_u, a_u^2/m_k=b_u^2$, which yields $m_k=1$, a
contradiction.
\end{proof}

The completeness properties of the cap $K_{m_1,m_2}^{(2s+1)}$ can
be now investigated.
\begin{lemma}\label{chiave0} Assume that $q>4$. The cap $K=K_{m_1,m_2}^{(2s+1)}$ covers all the points in
$PG(2s+1,q)$, with the exception of points
$$
(1,m,0,0\ldots,0),\,\,\,m\in \fq, m\neq 0,
$$
when $s>0$.
\end{lemma}
\begin{proof}
We prove the assertion by induction on $s$. The case $s=0$ is
trivial. Assume then that $s>0$.  Note that by inductive hypothesis
any point $P\in L_2$ is covered by the secants of $A_3^{(2s+1)}$,
provided that $P$ is not of type $(0,0,1,m,0\ldots,0),\,m\neq 0$. We
first prove that any $P= (0,0,1,m,0\ldots,0)$, $m\neq 0$, is
actually covered by the secants of $K$. If $s=1$, that is $P=
(0,0,1,m)$, then clearly
$$
P=(0,0,1,0)+m(0,0,0,1);
$$
otherwise,
$$
P=(1,m_1,a,a^2,0,\ldots,0)+(1,m_1,b,b^2,0,\ldots,0)
$$
for any $a,b$ such that $a+b=m$. This proves that all the points
in $L_2$ are covered by the secants of $K$.

We now consider points $P\in L_1\setminus L_2$. From Lemma
\ref{GI1} together with (\ref{link1}), it follows that if $P\notin
A_2^{(2s+1)}$, $P\neq (0,1,0\ldots,0)$, then $P$ is contained in
 $\frac{q-2}{2}$ distinct lines joining  two distinct
points in $A_2^{(2s+1)}\cup \{(0,1,0\ldots,0)\}$. The hypothesis
$q>4$ implies $\frac{q-2}{2}>1$, whence if $P\neq
(0,1,0\ldots,0)$, then $P$ is covered by the secants of
$A_2^{(2s+1)}$. Actually, also $P=(0,1,0\ldots,0)$ is covered by
the secants of $K$, as
$$
P=\frac{1}{m_1+m_2}\left((1,m_1,a_i,a_i^2,\ldots,a_s^2)+(1,m_2,a_i,a_i^2,\ldots,a_s^2)\right).
$$
Now let $P=(1,m,d_1,d_1',\ldots,d_s,d_s')\in PG(2s+1,q)\setminus
L_1$, with either $d_l\neq 0$ or $d_l'\neq 0$ for some $l$. Assume
that $m=m_j$, $j\in \{1,2\}$. Let $L^{(m_j)}$ be the hyperplane of
equation $X_2=m_jX_1$. Note that $L^{(m_j)}\setminus L_2$ is an
affine space $AG(2s,q)$, and that
$$
(K\cap (L^{(m_j)}\setminus L_2))\cup  \{(1,m_j,0,\ldots,0)\}=\cP^s.
$$
Then by Lemma \ref{GI1}, through every point
$P=(1,m_j,d_1,d_1',\ldots,d_s,d_s')\in L^{(m_j)}\setminus L_2$
with  $d_u'\neq d_u^2$ for some $u\in \{1,\ldots,s\}$, there pass
$\frac{q-2}{2}$ lines joining  two distinct points in $(K\cap
(L^{(m_j)}\setminus L_2))\cup  \{(1,m_j,0,\ldots,0)\}$. The
hypothesis $q>4$ ensures $\frac{q-2}{2}>1$, whence $P$ is covered
by the secants of $K\cap (L^{(m_j)}\setminus L_2)$.

Finally we consider points $P=(1,m,d_1,d_1',\ldots,d_s,d_s')\in
PG(2s+1,q)\setminus L_1$, $m\notin \{m_1,m_2\}$, with either
$d_u\neq 0$ or $d_u'\neq 0$ for some $u$. Let $l$ be the smallest
index for which either $d_l\neq 0$ or $d_l'\neq 0$. Let
$t_m=(m+m_1)/(m_1+m_2)$, $s_m=m+m_1$, $r_m=m+m_2$. {}From Lemma
\ref{GI2} together with (\ref{link1}), it follows that
$A_1^{(2s+1)}\cup \{(1,m_1,0\ldots,0),(1,m_2,0,\ldots,0)\}$ is a cap
which covers all the points in $PG(2s+1)\setminus L_1$. Therefore,
$P$ is covered by the secants of $A_1^{(2s+1)}$ provided that $P$
does not belong to $C_1\cup C_2$, where $C_j$ is the union of lines
joining $(1,m_j,0,\ldots,0)$ to some point in $A_1^{(2s+1)}$.

Straightforward computation yields that points in $(C_1\cup
C_2)\setminus L_1$ such that $X_2=mX_1$ are
$$
\left(1,m,\frac{m+m_j}{m_1+m_2}a_1,\frac{m+m_j}{m_1+m_2}a_1^2,\ldots,\frac{m+m_j}{m_1+m_2}a_s,\frac{m+m_j}{m_1+m_2}a_s^2\right),
$$
$a_1\in \fq,\,(a_1,\ldots,a_s)\neq (0,\ldots,0), \,\,j=1,2$. If $P$
is a point of this type, then either $d_l'= d_l^2/t_m$ or $d_l'=
d_l^2/(t_m+1)$, according to whether $j=1$ or $j=2$. Therefore $P$
is covered by the secants of $K\setminus L_1$ provided that
$d_l'\neq d_l^2/t_m$ and $d_l'\neq d_l^2/(t_m+1)$.

Now we establish whether $P$ belongs to a line joining a point in
$K\setminus L_1$ to a point in $K\cap L_1$. We first look for points
\begin{equation*}
P_1=(1,m_1,a_1,a_1^2,\ldots,a_s,a_s^2),\,\,
P_3=(0,1,b_1,b_1^2,\ldots,b_s,b_s^2)
\end{equation*}
such that $P=P_1+s_mP_3$, that is,
$$
 a_u+s_mb_u=d_u, \quad (a_u^2+s_mb_u^2)=d'_{u},\quad u=1,\ldots,s.
$$
The  system
$$
\left\{
\begin{array}{l}
a_u^2+s_m^2b_u^2=d_u^2\\
a_u^2+s_mb_u^2=d_u'
\end{array}
\right.
$$
has precisely one solution when $s_m\ne 1$, that is, $m\ne m_1+1$.
This solution corresponds to two points in $K$ provided that not all
$a_u$ and not all $b_u$ are equal to $0$, that is
\begin{equation*}
P\notin \left\{ (1,m,a_1,a_1^2,\ldots,a_s,a_s^2),(1,m,
s_mb_1,s_mb_1^2,\ldots,s_mb_s,s_mb_s^2)\right\}.
\end{equation*}
Next we look for points
$$
P_2=(1,m_2,a_1,a_1^2,\ldots,a_s,a_s^2),\,\,
P_3=(0,1,b_1,b_1^2,\ldots,b_s,b_s^2)
$$
such that $P=P_2+r_mP_3$.  Arguing as above, we deduce that  $P_2$
and $P_3$ exist, and both belong to $K$, provided that $m\ne
m_2+1$ and
\begin{equation*}
P\notin \left\{(1,m,a_1,a_1^2,\ldots,a_s,a_s^2),(1,m,
r_mb_1,r_mb_1^2,\ldots,r_mb_s,r_mb_s^2)\right\}.
\end{equation*}
To sum up, if $P$ is not covered by the secants of $K$ then all
the conditions (A), (B) and (C) below hold:
\begin{itemize}
\item[(A)] either $d'_{l}=d_{l}^2/t_m$ or
$d'_{l}=d_{l}^2/(1+t_m)$;

\item[(B)] either $m=m_1+1$, or $d'_{l}=d_{l}^2$ or
$d'_{l}=d_{l}^2/s_m$;

\item[(C)] either $m=m_2+1$, or $d'_{l}=d_{l}^2$ or
$d'_{l}=d_{l}^2/r_m$.
\end{itemize}
{}As either $d_l\neq 0$ or $d_l'\neq 0$, from (A) it follows that
actually both $d_l$ and $d_l'$ are different from $0$. Let
$\rho=\frac{d_l^2}{d'_l}$. Let $E_1=\{t_m,1+t_m\}$, $E_2=\{1,s_m\}$,
$E_3=\{1,r_m\}$.

 Assume first that $m\ne m_1+1$ and $m\ne m_2+1$. Then $\rho$ belongs to
 all sets $E_1$, $E_2$, $E_3$. But this is
 impossible as the intersection $E_1\cap E_2 \cap E_3$ is empty.

Assume now that $m=m_1+1$. Then $\rho\in E_1\cap E_3$. This yields
that either $r_m=t_m$ or $r_m=t_{m}+1$; that is, either
$$m_1+m_2+1=\frac{1}{m_1+m_2}\quad \text{ or }\quad
m_1+m_2=\frac{1}{m_1+m_2}.
$$
The former case yields that $(m_1+m_2)$ is a root of $T^2+T+1$;
the latter  $m_1+m_2=1$. Both are not possible as $(m_1+m_2)^3\ne
1$.

Finally, assume that $m=m_2+1$. Then $\rho\in E_1\cap E_2$, which
yields that either $s_m=t_m$ or $s_m=t_{m}+1$; that is, either
$$
m_1+m_2+1=\frac{m_1+m_2+1}{m_1+m_2}\quad \text{ or }\quad
m_1+m_2=\frac{m_1+m_2+1}{m_1+m_2} .
$$
A contradiction is then obtained as  for the case $m=m_1+1$. This
completes the proof.
\end{proof}
Now we are in a position to conclude the proof of Proposition
\ref{chiave}.
%
It is enough to note that in the proof of Lemma \ref{chiave0} the
point $(1,0,\ldots,0)$ is not used to prove that each point of
$PG(2s+1,q)$ different from $ (1,m,0,0\ldots,0) $ is covered by
the secant of $K_{m_1,m_2}^{(2s+1)}$.


\begin{thebibliography}{99}
\bibitem{ABA} V.~Abatangelo, A class of complete $[(q+8)/3]$-arcs of {$%
PG(2,q)$}, with $q=2^{h}$ and $h$ ($\geq 6$) even, \emph{Ars
Combin.}, Vol. 16 (1983), pp. 103--111.

\bibitem{P2} A.A. Davydov, G. Faina, S. Marcugini and F. Pambianco, Computer
search in projective planes for the sizes of complete arcs, \emph{J. Geom.},
Vol. 82 (2005), pp. 50-62.

\bibitem{DGMP} A.A. Davydov, M. Giulietti, S. Marcugini and F. Pambianco,
\textquotedblleft On sharply transitive sets in $PG(2,q)$,\textquotedblright%
\ \emph{Innovations Incidence Geom., }to appear.

\bibitem{P1} A.A. Davydov, S. Marcugini and F. Pambianco,
    Complete caps in projective spaces $PG(n,q)$, \emph{J.
    Geom.}, Vol. 80 (2004), no. 1-2, pp. 23--30.

\bibitem{DO} A.A. Davydov and P.R.J. \"Osterg\aa rd, Recursive
    constructions of complete caps, \emph{J. Statist. Planning
    Infer.}, 95 (2001), 167--173.


\bibitem{EDE} Y. Edel, Extensions of generalized product caps, \emph{Des.
Codes Cryptogr.}, Vol. 31 (2004), no. 1, pp. 5--14.

\bibitem{FainaJiul} G.~Faina and M.~Giulietti, On small dense arcs in Galois
planes of square order, \emph{Discrete Math.} \textbf{267} (2003), 113--125,
(Gaeta, 2000).

\bibitem{GDT} E.M. Gabidulin, A.A. Davydov and L.M. Tombak, Linear codes
with covering radius $2$ and other new covering codes, \emph{IEEE Trans.
Inform. Theory}, Vol. 37 (1991), no. 1, pp. 219--224.

\bibitem{Gi1} M. Giulietti, Small Complete Caps in Galois Affine Spaces,
\emph{J. Algebraic Comb.}, Vol. 25 (2007), no. 2, 149--168.

\bibitem{Gi} M. Giulietti, Small complete caps in $PG(N,q)$, $q$ even, \emph{%
J. Comb. Des.}, Vol. 15 (2007), no. 5, pp. 420--436.

\bibitem{GiulPast} M. Giulietti and F. Pasticci, Quasi-Perfect Linear Codes
with Minimum Distance 4, \emph{IEEE Trans. Inform. Theory}, Vol. 53 (2007),
no. 5, pp. 1928--1935.

\bibitem{Hirs} J.\ W.\ P. Hirschfeld, Projective Geometries over Finite
Fields, second edition, Oxford University Press, Oxford, 1998.

\bibitem{HS1} J.W.P. Hirschfeld and L. Storme, The packing problem in
statistics, coding theory and finite projective spaces, \emph{J. Statist.
Planning Infer.}, Vol. 72 (1998), no. 1-2, pp. 355--380.

\bibitem{HS} J.W.P. Hirschfeld and L. Storme, The packing problem in
statistics, coding theory and finite projective spaces: update 2001,
Blokhuis, A. (ed.) et al., Finite geometries. Proceedings of the fourth Isle
of Thorns conference, Brighton, UK, April 2000. Dordrecht: Kluwer Academic
Publishers. Dev. Math., Vol. 3 (2001), pp. 201-246.

\bibitem{KV} J.H. Kim and V. Vu, Small Complete Arcs in Projective Planes,
\emph{Combinatorica}, Vol. 23 (2003), pp. 311--363.

\bibitem{MWS} F.J. MacWilliams and N.J.A. Sloane, The Theory of
    Error-Correcting Codes, Amsterdam, Netherlands: North-Holland, 1977.

\bibitem{MUK} A.C. Mukhopadhyay, Lower bounds on $m_t(r,s)$, \emph{J.
Combin. Theory Ser. A}, Vol. 25 (1978), no. 1, pp. 1--13.

\bibitem{OS} P.R.J. \"{O}sterg\aa rd, Computer search for small complete
caps, \emph{J. Geom}, Vol. 69 (2000), no. 1-2, pp. 172--179.

\bibitem{PS} F. Pambianco and L. Storme, Small complete caps in spaces of
even characteristic, \emph{J. Combin. Theory Ser. A}, Vol. 75 (1996), no. 1,
pp. 70--84.

\bibitem{PamSt-unpub} F.\ Pambianco and L.\ Storme, unpublished manuscript
(1995).

\bibitem{segB59b} B.~Segre, On complete caps and ovaloids in
three-dimensional {G}alois spaces of characteristic two, \emph{Acta Arith.},
Vol. 5 (1959), pp. 315-332.

\bibitem{SZ} T.~Sz{\H{o}}nyi, Small complete arcs in {G}alois planes, \emph{%
Geom. Dedicata}, Vol. 18 (1985), pp. 161--172.
\end{thebibliography}
\end{document}